\documentclass[10pt]{amsart}

\usepackage{amssymb,amsthm,amsmath}
\usepackage[numbers,sort&compress]{natbib}
\usepackage{color}
\usepackage{graphicx}
\usepackage{tikz}

\title[ ]{Reducibility of the quantum harmonic oscillator in $d$-dimensions with  finitely differentiable perturbations  }

\author{Wenwen Jian}
\address[Wenwen Jian]{School of Mathematical Sciences,
Fudan University,
Shanghai 200433,
P. R. China} \email{wwjian16@fudan.edu.cn}

\DeclareMathOperator{\meas}{meas}

\newcommand{\C}{\mathbb{C}}
\theoremstyle{plain}
\newtheorem{thm}{Theorem}[section]
 \newtheorem{cor}[thm]{Corollary}
 \newtheorem{lem}[thm]{Lemma}
 \newtheorem{prop}[thm]{Proposition}
 \theoremstyle{definition}
 \newtheorem{defn}[thm]{Definition}
 \theoremstyle{remark}
 \newtheorem{rem}[thm]{Remark}
 
 \numberwithin{equation}{section}
\begin{document}

\begin{abstract}
In this paper, the  $d$-dimensional quantum harmonic oscillator with a pseudo-differential time quasi-periodic perturbation
 \begin{equation}\label{0}
\text{i}\dot{\psi}=(-\Delta+V(x)+\epsilon W(\omega t,x,-\text{i}\nabla))\psi,\ \ \ \ \  x\in\mathbb{R}^d
\end{equation}
is considered, where $\omega\in(0,2\pi)^n$, $V(x):=\sum_{j=1}^d v_j^2x_j^2, v_j\geq v_0>0$, and $W(\theta,x,\xi)$ is a real polynomial in $(x,\xi)$ of degree at most two, with coefficients belonging to $C^{\ell}$ in $\theta\in\mathbb{T}^n$ for the order $\ell$ satisfying $\ell\geq 2n-1+\beta,\ 0<\beta<1$.
\emph{}
Using techniques developed by Bambusi-Gr\'ebert-Maspero-Robert [\emph{{Anal. PDE. 11(3):775-799, 2018}}] and R\"ussmann [\emph{pages 598-624. Lecture Notes in Phys., Vol. 38, 1975}], the paper shows that for any $|\epsilon|\leq \epsilon_{\star}(n,\ell)$, there is a set $\mathcal{D}_{\epsilon}\subset (0,2\pi)^n$ with big Lebesgue measure, such that for any $\omega \in\mathcal{D}_{\epsilon}$, the system (\ref{0}) is reducible.
\end{abstract}

\maketitle

\section{Introduction and the main result}
Recently, Bambusi-Gr\'ebert-Maspero-Robert \cite{Bambusi2018reducibility} proved a reducibility result for a  $d$-dimensional quantum harmonic oscillator with a pseudo-differential time quasi-periodic perturbation:
\begin{equation}\label{eq}
\text{i}\dot{\psi}=H_{\epsilon}(\omega t)\psi,\ \ \ \ \  x\in\mathbb{R}^d,
\end{equation}
where
\begin{equation}\label{H}
H_{\epsilon}(\omega t):=H_0+\epsilon W(\omega t,x,-\text{i}\nabla),\ \ \ \ \ H_0:=-\Delta+V(x),
\end{equation}
$\text{i}=\sqrt{-1}$ is the unit imaginary number, $\omega$ are parameters  belonging to the set $\mathcal{D}=(0,2\pi)^n$. Moreover, $V$ and $W$ satisfy the following two conditions:
\begin{description}
  \item[C1] \begin{equation*}
 V(x):=\sum_{j=1}^d v_j^2x_j^2,\ \ \ \ \ v_j\geq v_0>0,\ \ \ \ j=1,\cdots,d,
\end{equation*}
  \item[C2] $W(\theta,x,\xi)$ is a real polynomial in $(x,\xi)$ of degree at most two, with coefficients being analytic in $\theta\in\mathbb{T}^n:=(\mathbb{R}/2\pi\mathbb{Z})^n$.
\end{description}
Note that for $\epsilon=0$, the spectrum of $H_0$ defined in  (\ref{H}) is
\begin{equation*}
  \sigma(H_0)=\{\lambda_k\}_{k\in\mathbb{N}^d},\ \ \ \ \lambda_k\equiv\lambda_{(k_1,\cdots,k_d)}:=\sum_{j=1}^d (2k_j+1)v_j,\ \ \ \ k\in\mathbb{N}^{d},
\end{equation*}
which is dense in $\mathbb{R}$.  The paper \cite{Bambusi2018reducibility} is the first reducibility result for a system in which  the  gap of the unperturbed spectrum is dense in $\mathbb{R}$.

The reducibility result in  \cite{Bambusi2018reducibility} is obtained by using pseudo-differential calculus together with  KAM techniques. More precisely, for the perturbation $W$ defined by condition \textbf{C2}, the correspondence between quantum dynamics of quadratic Hamiltonian $H_{\epsilon}(\omega t)$ defined in (\ref{H})  and classical Hamiltonian
\begin{equation}\label{classical}
h'_{\epsilon}(\omega t,x,\xi):=\sum_{j=1}^d(\xi_j^2+ v_j^2x_j^2)+\epsilon W(\omega t,x,\xi)
\end{equation}
is exact (without error term), and then the reducibility of the quantum Hamiltonian $H_{\epsilon}(\omega t)$ is equivalent to the reducibility of the classical Hamiltonian (\ref{classical}).
Therefore, the result in  \cite{Bambusi2018reducibility} can be proved by exact quantization of the classical KAM theory which ensures reducibility of the classical Hamiltonian system (\ref{classical}). The advantage of the thought in  \cite{Bambusi2018reducibility} is that the reduction problem of a infinite-dimensional system with a unbounded perturbation can be turned into the reduction problem of a classical finite-dimensional Hamiltonian system with bounded perturbation.

Motivated by the thought of \cite{Bambusi2018reducibility}, this paper extends the result of \cite{Bambusi2018reducibility} to $C^{\ell}$-perturbation with respect to $\theta\in\mathbb{T}^n$:
\begin{description}
\item[C3] $W(\theta,x,\xi)$ defined in (\ref{H}) is a real polynomial in $(x,\xi)$ of degree at most two, with coefficients belonging to the class $C^{\ell}$ in $\theta\in\mathbb{T}^n:=(\mathbb{R}/2\pi\mathbb{Z})^n$, where
  \begin{equation}\label{ell0}
    \ell\geq 2n-1+\beta,\ \ \ \ \  0<\beta<1.
  \end{equation}
\end{description}

\begin{rem}
The search for  persistence of KAM tori and reduction problems of finitely differentiable Hamilton's systems goes back to the work of Kolmogorov-Arnold-Moser, and attracts great attention over years.  For the Hamilton's equations
\begin{equation*}
  \dot{x}=H_y(x,y),\ \ \ \ \ \dot{y}=-H_x(x,y), \ \ \ \ \ (x,y)\in\mathbb{R}^{2n},
\end{equation*}
where $H\in C^{\ell}$, it was shown by P\"oschel \cite{poschel1982,Salamon2004The} and Cheng \cite{cheng2011Non} that the hypotheses $\ell>2n$ is \emph{optimal} for the existence of KAM torus.

When considering the Hamiltonian \begin{equation}\label{low}
h_{\epsilon}(I,\theta,z,\bar{z})=h_0(I,z,\bar{z})+\epsilon q(I,\theta,z,\bar{z}), \ \ \ \ \ (I,\theta,z,\bar{z})\in\mathbb{R}^n\times\mathbb{T}^n\times\mathbb{C}^{2d},
\end{equation}
where  $q$ is  $C^{\ell}$ in $\theta$,
 Chierchia-Qian \cite{Chierchia2004} proved the persistence  of lower dimensional  tori with $\ell>6n+5$ and Sun-Li-Xie \cite{Li2018Reducibility} considered the reduction problem with $\ell\geq 200n$.
In \cite{yuan2013}, Yuan-Zhang established a reduction theorem for the time dependent Schr\"odinger equation with a $C^{\ell}$-perturbation where $\ell\geq 100(3n+2\tau+1)$.

Obviously, the orders $\ell$  appeared in \cite{Chierchia2004,Li2018Reducibility,yuan2013} are not optimal. In \cite{Chierchia2004}, Chierchia-Qian pointed out  that
\begin{itemize}
  \item[]It would be interesting to find the optimal value: for example, is it true that the theorem for the persistence of lower dimensional tori holds provided $\ell>2n$ (as the maximal case)?
\end{itemize}
In the  present paper,  the lower bounds of $\ell$ is reduced to $2n-1+\beta,\ 0<\beta<1$, which improves the exiting regularity results to some extent.   This is based on  R\"ussimann's optimal estimation techniques \cite{russmann1975}.
\end{rem}

The  main theorem of this paper is
\begin{thm}\label{main}
Let $\psi$ be a solution of (\ref{eq}) and (\ref{H}) with conditions \textbf{C1} and \textbf{C3}. For any given $0<\gamma\ll 1$, there exist $\epsilon_{\star}$ with $0<\epsilon_{\star}=\epsilon_{\star}(n,\gamma)\ll \gamma^8$, such that for any  $0<|\epsilon|<\epsilon_{\star}$, there is a subset $\mathcal{D}_{\epsilon}\subset (0,2\pi)^{n}$  with
\begin{equation*}
  \meas((0, 2\pi)^n \setminus \mathcal{D}_{\epsilon}) \leq C\gamma^{\frac{1}{2}},
\end{equation*}
 and for any $ \omega\in\mathcal{D}_{\epsilon}$,  there is a unitary (in $L^2$) time quasi-periodic map $U_{\omega}(\omega t)$ such that defining $\psi'$ by $U_{\omega}(\omega t)\psi'=\psi$, it satisfies the equation
\begin{equation}\label{hinfty}
  \emph{i}\dot{\psi}'=H_{\infty}\psi',
\end{equation}
with  a  time independent operator
\begin{equation*}
 H_{\infty}=-\Delta+V_{\infty}(\omega)+e_{\infty},\ \ \ \ \ V_{\infty}(\omega)=\sum_{j=1}^d (v_j^{\infty})^2x_j^2,
\end{equation*}
where $e_{\infty}=e_{\infty}(\omega)$ and $v_j^{\infty}=v_j^{\infty}(\omega)$ are defined for $\omega\in\mathcal{D}_{\epsilon}$ and fulfill the estimates
\begin{equation*}
  |e_{\infty}|\leq\epsilon^{ \frac{1}{2}},\ \ \ \ \  |v_j-v_j^{\infty}|\leq \epsilon^{\frac{1}{2}},\ \ \ \ \ j=1,\cdots, d.
\end{equation*}
Finally, the following properties hold
\begin{enumerate}
  \item $\forall r\geq 0, \ \forall \psi\in \mathcal{H}^r$, $\theta\mapsto U_{\omega}(\theta)\in C^0(\mathbb{T}^n;\mathcal{H}^r)$;
  \item $\forall r\geq 0$, $\exists C_r>0$ such that for all $\theta\in\mathbb{T}^n$,
  \begin{equation}\label{u}
    \|U_{\omega}(\theta)-id\|_{\mathcal{L}(\mathcal{H}^{r+2},\mathcal{H}^r)}\leq C_r \epsilon^{\frac{1}{8}};
  \end{equation}
  \item $\forall r\geq 0$ and $\forall\ 0\leq l\leq \ell$, the map $\theta\mapsto U_{\omega}(\theta)$ is of class $C^l(\mathbb{T}^n,\mathcal{L}(\mathcal{H}^{r+4l+2},\mathcal{H}^r))$,
\end{enumerate}
where  the weighted Sobolev space $\mathcal{H}^r$ is defined as
\begin{equation*}
  \mathcal{H}^{r}=\left\{\psi\in L^2(\mathbb{R}^{d}):H_0^{\frac{2}{r}}\psi\in L^2(\mathbb{R}^d)\right\}, \ \ r\geq 0
\end{equation*}
with the norm $\|\psi\|_{r}=\left\|H_0^{\frac{2}{r}}\psi\right\|_{L^2(\mathbb{R}^d)}$.
\end{thm}

\begin{rem}
This reduction result is obtained by KAM iteration, pseudo-differential calculus and analytic approximation techniques. The proof consists of three steps. Firstly,  pseudo-differential calculus properties are used in order to transform the original system (\ref{eq}) to a system with a differentiable perturbation of $\theta\in\mathbb{T}^n$. Secondly,  the analytic approximation lemma is introduced to deal with the differentiable perturbation. Finally, by applying KAM techniques, the reduction result is obtained.
\end{rem}

\begin{rem}
The KAM scheme in this paper is non-standard to ensure that the iterative process can constantly continue. From analytic approximation, the perturbation $q_0$ has the expression (see subsection \ref{appro} for details)
 $$q_0=q_{0}^{(0)}+\sum_{m=1}^{\infty}\left(q_0^{(m)}-q_0^{(m-1)}\right)=\sum_{m=0}^{\infty}\epsilon_mq_{0,m},$$
where $q_0^{(m)}$  are analytic in $\mathbb{T}^n_{\sigma_m}$ and $q_{0,m}$ are of size $\mathcal{O}(\epsilon_m)$. In the present paper, the author constructs a series of real-analytic symplectic transformations $\widetilde{\Phi}_m,\ m=0,1,\cdots$, so that
\begin{equation*}
  \left(h_0+q_{0}^{(m)}\right)\circ\widetilde{\Phi}_m=h_{m+1}+\epsilon_{m+1}q_{m+1},
\end{equation*}
where the sequence of $h_m$'s is in ``normal form" while the perturbations $\epsilon_{m}q_{m}$ are real-analytic functions in $\mathbb{T}^n_{s_{m}} (s_m< \sigma_m)$ of smaller and smaller size $\mathcal{O}(\epsilon_{m})$. The parameters $\omega$ appeared in $h_{m}$  vary in smaller and smaller compact sets $\mathcal{D}_{m+1}\subset\mathcal{D}_m$ of relatively large Lebesgue measures.
The symplectic transformations $\widetilde{\Phi}_m$ are seeked of the form $$\widetilde{\Phi}_m=\Phi_0\circ\Phi_1\circ\cdots\circ\Phi_m. $$
Thus by induction, one has
\begin{equation}\label{equation}
  (h_m+\epsilon_mq'_m)\circ\Phi_m=h_{m+1}+\epsilon_{m+1}q_{m+1},
\end{equation}
where
\begin{equation*}
q'_0=q_{0,0},\ \ \ \ \  q'_m= q_v+q_{0,m}\circ\widetilde{\Phi}_{m-1},\ \ \ \ m=1,2,\cdots.
\end{equation*}
$\epsilon_mq'_m$ and $\epsilon_mq_m$ are of the same size, and (\ref{equation}) fits  in more standard KAM approaches.
A remark is that in order for this approach to work, the maps $\widetilde{\Phi}_m$ have to verify  suitable compatibility relations with respect to the analytic domains. More precisely,
\begin{equation*}
  \Phi_m: \mathbb{T}^n_{s_{m+1}}\times \mathbb{C}^{2d}\rightarrow\mathbb{T}^n_{s_m}\times \mathbb{C}^{2d},\ \ \ \ \
  \widetilde{\Phi}_{m-1}: \mathbb{T}^n_{s_{m}}\times \mathbb{C}^{2d}\rightarrow\mathbb{T}^n_{\sigma_m}\times \mathbb{C}^{2d},\ \ \ \ \  m=0,1,\cdots.
\end{equation*}
\end{rem}

\begin{rem}
The difficulty of the present paper is that the perturbation $W$ has quietly low differentiability hypotheses.
For a differentiable perturbation $W$ of order $\ell$, the condition that  $\ell$ is sufficiently large is necessary so as to guarantee the convergence of the KAM iterations. Generally, at the $m$-th KAM step, all the new perturbation terms should be  $\mathcal{O}(\epsilon_{m+1})$.
In order to obtain  smaller lower bounds of $\ell$, more refined estimates are needed.  By applying some techniques of R\"ussmann  \cite{russmann1975},   this paper gets a better estimate of the Hamiltonian $f_m$ at the $m$-th KAM step:
\begin{equation*}
  [f_m]_{s^1_{m+1},\mathcal{D}_{m+1}}\leq C(n,\tau,\gamma_m)[q'_{m}]_{s_{m},\mathcal{D}_{m}}(s_m-s^1_{m+1})^{-\tau},
\end{equation*}
where the exponent $\tau$ of $(s_m-s^1_{m+1})^{-1}$ in this inequality is optional ($\tau>n-1$ appears in the Diophantine conditions, see subsection 3.3 for details). This optimal estimate implies  better estimations  of  new perturbations  and finally a smaller lower bound of $\ell$ (see subsections 3.2-3.3 for details).
\end{rem}

\begin{rem}
A difference between an analytic perturbation and the differentiable one is  the widths of  angle variables $\theta$ after the KAM iteration.
 If the perturbation $W$ is analytic  in some strip domain $|\Im\theta|\leq r_m$, then the strip widths $r_m,m=0,1,\cdots$, have a uniform non-zero lower bound $\frac{r_0}{2}$. When the perturbation $W$ is of class $\C^{\ell}$, by using Jackson-Moser-Zehnder's approximation lemma (see Lemma \ref{smooth}), the  new Hamiltonians  are still analytic in $|\Im\theta|\leq s_m$ at the $m$-th KAM step. However, the strip widths have no non-zero lower bound. Actually,  by choosing $s_m=\frac{1}{2}\epsilon_{m+1}^{\frac{1}{\ell}}$ of this paper, the sequence $s_m$  goes to zero very rapidly.
\end{rem}

Denote $\mathcal{U}_{\epsilon,\omega}(t,\kappa)$ the propagator generated by (\ref{eq}) such that $\mathcal{U}_{\epsilon,\omega}(\kappa,\kappa)=1,\forall \kappa\in\mathbb{R}$. An immediate consequence of Theorem \ref{main} is the Floquet decomposition:
\begin{equation}\label{propagator}
  \mathcal{U}_{\epsilon,\omega}(t,\kappa)=U_{\omega}^{*}(\omega t)e^{-\text{i}(t-\kappa)H_{\infty}}U_{\omega}(\omega \kappa).
\end{equation}
 An other consequence of (\ref{propagator}) is that for any $r>0$, the norm $\|\mathcal{U}_{\epsilon,\omega}(t,0)\psi_0\|_r$ is bounded uniformly in time:

\begin{cor}\label{hsmorm}
Let $\omega\in\mathcal{D}_{\epsilon}$ with $0<\epsilon<\epsilon_{\star}$. The following is true: for any $r>0$ one has
\begin{equation}\label{hsmorn1}
  c_r\|\psi_0\|_r\leq \|\mathcal{U}_{\epsilon,\omega}\psi_0\|_r\leq C_r\|\psi_0\|_r,\ \ \ \ \  \forall t\in\mathbb{R},\ \  \ \ \forall \psi_0\in\mathcal{H}^r,
\end{equation}
for some $c_r>0,C_r>0$. Moreover, there exists a constant $c'_r$, such that if the initial date $\psi_0\in\mathcal{H}^{r+2}$, then
\begin{equation*}
  \|\psi_0\|_{r}-\epsilon c'_r\|\psi_0\|_{r+2}\leq \|\mathcal{U}_{\epsilon,\omega}\psi_0\|_r\leq \|\psi_0\|_{r}+\epsilon c'_r\|\psi_0\|_{r+2},\ \ \ \ \ \forall t\in\mathbb{R}.
\end{equation*}
\end{cor}

Denote by $\{\phi_{k}\}_{k\in\mathbb{N}^d}$ the set of Hermite functions, namely the eigenvectors of $H_0:H_0\phi_k=\lambda_k\phi_k$. They form an orthonormal basis of $L^2 (\mathbb{R}^d)$, and writing $\phi=\sum_k c_k\phi_k$ one has $\|\phi\|_r^2\simeq\sum_k(1+|k|)^{2r}|c_k|^2$. Denote $\psi(t)=\sum_{k\in\mathbb{N}^d}c_k(t)\phi_k$ the solution of (\ref{eq}) written on the Hermite basis. Then (\ref{hsmorn1}) implies the following dynamical localization for  the energy of the solution:$\forall r\geq 0, \exists C_{r}\equiv C_r(\psi^0)>0$ such that
$$\sup_{t\in\mathbb{R}}|c_r(t)|\leq C_r(1+|k|)^{-r},\ \ \ \ \ \ \forall k\in\mathbb{N}^d.$$

From the dynamical property, one obtains easily that every state $\psi\in L^2(\mathbb{R}^d)$ is a bounded state for the time evolution $\mathcal{U}_{\epsilon,\omega}(t,0)\psi$ under the conditions of Theorem \ref{main} on $(\epsilon,\omega)$. The corresponding definitions are given in \cite{ev83}:
\begin{defn}[see \cite{ev83}]
A function $\psi\in L^2(\mathbb{R}^d)$ is a bounded state (or belongs to the point spectral subspace of $\{\mathcal{U}_{\epsilon,\omega}(t,0)\}_{t\in\mathbb{R}}$) if the quantum trajectory $\{\mathcal{U}_{\epsilon,\omega}(t,0)\psi:t\in\mathbb{R}\}$ is a precompact subset of $L^2(\mathbb{R}^d)$.
\end{defn}

\begin{cor}\label{boundstate}
Under the conditions of Theorem \ref{main} on $(\epsilon,\omega)$, every state $\psi\in L^2(\mathbb{R}^{d})$ is a bounded state of $\{\mathcal{U}_{\epsilon,\omega}(t,0)\}_{t\in\mathbb{R}}$.
\end{cor}

\begin{cor}
The operator $\mathcal{U}_{\omega}$ induces a unitary transformation in $L^2(\mathbb{T}^n) \otimes L^2(\mathbb{R}^d)$ which transforms the Floquet operator
$$K:=-\emph{i}\omega\cdot\frac{\partial}{\partial\theta}+H_0+\epsilon W(\theta)$$
into
$$-\emph{i}\omega\cdot\frac{\partial}{\partial\theta}+H_{\infty}.$$
Thus one has that the spectrum of $K$ is pure point  and its eigenvalues are
 $$\omega\cdot k+\lambda_j^{\infty}:=\sum_{l=1}^{d}\omega_lk_l+\sum_{l=1}^d(2j_l+1)v_l^{\infty}+e^{\infty},\ \ \ \ j\in\mathbb{N}^{d},\ \ \ \ k\in\mathbb{Z}^n.$$
\end{cor}

 Let us recall the history of reduction problems.  For a time dependent linear system
\begin{equation}\label{redu}
  \dot{x}=A(\theta)x,\ \ \ \ \ \dot{\theta}=\omega,
\end{equation}
where $A(\theta)$ is a $m\times m$ real-valued or complex-valued analytic matrix defined on $\mathbb{T}^n$. If $A(\theta)$ is periodic (i.e., $n=1$) and continuous, it follows from Floquet theory that there exists a $gl(m,\mathbb{C})$-valued function $P$ defined on $\mathbb{T}$ such that the  change of variables
\begin{equation}\label{change}
  x=P(\theta)y
\end{equation}
transform  (\ref{redu}) into a  linear system with constant coefficient
\begin{equation}\label{cons}
  \dot{y}=By.
\end{equation}
For a time quasi-periodic coefficient system (\ref{redu}), Johnson-Sell \cite{Johnson1981} proved that if the quasi-periodic coefficient matrix $A(\omega t)$ satisfies a ``full spectrum" assumption, then the system (\ref{redu}) is reducible.
However, for a general time quasi-periodic matrix $A(\omega t)$, the change (\ref{change}) usually does not exist, see \cite{MR1648125}.
Let us consider a special case
\begin{equation}\label{special}
\dot{x}=( L + \epsilon Q(\omega t))x,
\end{equation}
where $L$ is a $m\times m$ constant matrix, $Q(\theta)$ is an analytic $m\times m$ matrix defined on $\mathbb{T}^n$,  the  frequencies $\omega$  are rational independent and $\epsilon$ is small. In this case, the well-known KAM (Kolmogorov-Arnold-Moser) theory can be applied. For example,
Jorba-Sim\'o \cite{MR1168974} considered the reduction problem of (\ref{special}), where $L$ is constant matrix with different eigenvalues and  Xu \cite{xu1999reduci} considered the reduction problem of (\ref{special}), where $L$ is constant matrix with multiple eigenvalues. Eliasson \cite{Eliasson1992} considered the one-dimensional Schr\"odinger equation $-\ddot{x}=(E-q(\omega t))x$.

For the case of PDEs, one of the popular models is the time quasi-periodic Schr\"odinger equation
\begin{equation}\label{schroding}
   \text{i}\dot{\psi}=(A+\epsilon P(\omega t,x,-\text{i}\partial_x))\psi, \ \ \ \ \ x\in \mathbb{R}, \ \ \ \ \psi(t)\in \mathcal{H},
\end{equation}
where $\mathcal{H}$ is some separable Hilbert space, $A$ is a  positive self-adjoint (unbounded) operator, and the perturbation $P$ is time quasi-periodic and it may or may not depend on $x$ or (and) $\partial_x$.
It is well-known that the long-time behavior of the solution $\psi$ of the time dependent Schr\"odinger equation (\ref{schroding}) is closely related to the spectral properties of the Floquet operator
\begin{equation*}
  K_{F}:=-\text{i}\omega\cdot\partial_{\theta}+A+\epsilon P,\ \ \ \ \mathcal{H}\otimes L^{2}(\mathbb{T}^n).
\end{equation*}
If the Floquet operator $K_F$  is of pure point spectrum or no absolutely continuous spectrum under some conditions, then the reducibility of the system  (\ref{schroding}) can be obtained, see \cite{Blekher1992,Bellissard1985,Combescure1987,Duclos1996,Graffi2000,Howland1989,nenciu1993} for example.

The techniques from analytic KAM theory for PDEs  were first used by Bambusi-Graffi  \cite{Bambusi2001time} to  reduce the equation (\ref{schroding}), where the perturbation $P$ is some analytic unbounded operator. After that, Liu-Yuan  \cite{liu2010spectrum}  proved a reduction theorem of the original quantum Duffing oscillator. Particularly, for the quantum harmonic oscillator (i.e., $A=-\partial_{xx}+x^2$ in (\ref{schroding})),  reduction problems of the system (\ref{schroding}) with bounded perturbations are considered, for example \cite{wang08pure,grebert2011kam,wang2017reducibility}.
Recently, Bambusi \cite{Bambusi2017reducibility1,Bambusi2017reducibility2} studied the  reduction problem of (\ref{schroding}), where $A=-\partial_{xx}+V(x)$ with $V$ being a polynomial potential of degree $2l\ (l\geq 1)$ and the  perturbations are related to $\partial_x$. By using pseudo-differential calculus together with  KAM techniques, Bambusi proved that the system (\ref{schroding}) is reducible when $W$ has controlled growth at infinity.

For the high-dimensional time quasi-periodic Schr\"odinger operator, the perturation results are few up to now. Eliasson-Kuksin \cite{EliassonKuksin2009on} dealt with the equation (\ref{eq}) where $H_{\epsilon}=\Delta-\epsilon W(\theta_0+\omega t,x;\omega), x\in\mathbb{R}^d$.
Paturel-Gr\'ebert \cite{Paturel2016On} studied the equation (\ref{eq}) where $H_{\epsilon}=-\Delta+|x|^2+\epsilon W(\omega t,x)$ with $W(\theta,x)$ belonging to $\mathcal{H}^s \ (s>d/2)$ with respect to $x\in\mathbb{R}^d$.
Lately, Bambusi-Gr\'ebert-Maspero-Robert \cite{Bambusi2018reducibility} developed the ideas of \cite{Bambusi2017reducibility2,Bambusi2017reducibility1} and  obtained a reducibility result for (\ref{eq}) and (\ref{H}) by using the exact correspondence between classical and quantum dynamics of quadratic Hamiltonians, which was also used in \cite{Hagedorn1986Non}.

Things are  more complicated for differentiable case  than analytic one.
In 1962, Moser considered the so-called twist mappings of an annulus and the twist mapping is assumed to be $C^{333}$ by using a smoothing technique via convolutions \cite{moser1962,moser1961,moser1966}. This number was brought down to 5 by R\"ussmann \cite{russmann1970} and $3-\beta\ (0<\beta<1)$ by Herman \cite{Herman}.
Later, the smoothing technique used in \cite{moser1962,moser1961} was improved, and the improved technique is based on the qualitative property of differentiable  functions, which can be characterized in terms of quantitative estimates for  approximating sequences of analytic functions, see \cite{ moser1966, moser1970, zehnder1975,zehnder1976}. After that, this kind of  analytic approximation is used to deal with  the  quasi-periodic solutions  and reduction problems of differentiable systems.
With respect to Hamiltonian systems, Moser \cite{moser1970,moser1966}  proved the existence of solutions of the Hamilton's equations
\begin{equation*}
  \dot{x}=H_y(x,y),\ \ \ \ \ \ \dot{y}=-H_x(x,y), \ \ \ \ \ \  (x,y)\in\mathbb{R}^{2n}
\end{equation*}
under the assumption that $H\in C^{\ell}$ with $\ell>2n+2$.  With better estimates for the solutions of linear partial differential equations, P\"oschel improved Moser's result to  $\ell>2n$, see \cite{poschel1982,Salamon2004The}. In \cite{cheng2011Non}, Cheng  proved the non-existence of KAM torus if $\ell<2n$, which
shows that the assumption $\ell>2n$ in \cite{poschel1982} is optimal.

In this paper,  the reducibility of  (\ref{eq}) and (\ref{H}) is equivalent to the reducibility of a Hamiltonian system with Hamiltonian
\begin{equation}\label{low}
h_{\epsilon}(\theta,z,\bar{z})=h_0(z,\bar{z})+\epsilon q(\theta,z,\bar{z}), \ \ \ \ \ (\theta,z,\bar{z})\in\mathbb{T}^n\times\mathbb{C}^{2d},
\end{equation}
where  $h_0$ is analytic in all variables and $q$ is a $C^{\ell}$-function in $\theta$.
In \cite{Chierchia2004}, Chierchia-Qian considered the Hamiltonian (\ref{low}) and obtained  the persistence and the regularity of lower
dimensional elliptic tori with $\ell>6n+5$.
To prove the reducibility of   wave equations,  Sun-Li-Xie \cite{Li2018Reducibility} considered (\ref{low}) with $\ell\geq 200n$.
In \cite{yuan2013}, Yuan-Zhang established a reduction theorem for the time dependent Schr\"odinger equation with a $C^{\ell}$-perturbation where $\ell\geq 100(3n+2\tau+1)$.

In fact, by doing more detailed estimations, the  smoothness assumptions in \cite{Chierchia2004,Li2018Reducibility,yuan2013}  can be improved. In this paper, a much better lower bound  $\ell\geq 2n-1+\beta,\ 0<\beta<1$ is obtained such that the Hamiltonian system with the  Hamiltonian (\ref{low}) is reducible. This is achieved by using R\"ussmann's methods  and KAM iterative techniques.

The paper is organized as follows. In section 2, the exact connection of  system (\ref{eq}) and  differentiable  Hamiltonian (\ref{hepsilon}) is obtained by using pseudo-differential calculus, and Theorem \ref{main} is obtained by Theorem \ref{kam}.
In subsection 3.1, analytic approximation lemma of $C^{\ell}$ functions is given. In subsection 3.2, the general strategy of differentiable KAM is given.  Homological equations arising from the first KAM step are obtained in subsection 3.3, the coordinates transformation and estimates of  new error terms are given in subsection 3.4.  The iterative lemma is received in subsection 3.5. Finally, Theorem \ref{kam} is proved in subsection 3.6.

\section{Proof of Theorem \ref{main}}
To start with, defining $x'_j=\sqrt{v_j}x_j$ to scale the variables $x_j$. Without loss of generality, this paper uses $x_j$ to express $x'_j$.
The un-perturbed operator defined in (\ref{H}) has the form
\begin{equation}\label{h01}
  H_0= \sum_{j=1}^{d}v_j \left( -\left(\frac{\partial}{\partial x_j}\right)^2 +  {x_j}^2  \right).
\end{equation}
Moreover, let
\begin{equation}\label{h01'}
h'_0(x,\xi)=\sum_{j=1}^d v_j(x_j^2+\xi_j^2).
\end{equation}

The Weyl quantization is considered in the following. The Weyl quantization of a symbol $f$ is the operator $\text{Op}^{w}(f)$ defined as usual as
\begin{equation*}
  \text{Op}^{w}(f)u(x)=\frac{1}{(2\pi)^d}\int_{y,\xi\in\mathbb{R}^d}e^{\text{i}(x-y)\xi}f(\frac{x+y}{2},\xi)u(y)\text{d}y\text{d}\xi.
\end{equation*}
We say that an operator $\text{Op}^{w}(f)$ is a Weyl operator with Weyl symbol $f$. It is easy to cheek that for polynomials $f$ of degree at most 2 in $(x,\xi)$, $\text{Op}^{w}(f)=f(x,D)+\text{const}$ where $D=\text{i}^{-1}\nabla_{x}$. The present paper uses the notation defined in \cite{Bambusi2018reducibility}: $f^{w}(x,D):=\text{Op}^{w}(f)$. Particularly in (\ref{H}) and (\ref{h01}), $W(\omega t,x,-\text{i}\partial_x)$ denotes the Weyl operator $W^{w}(\omega t,x,D)$ with the symbol $W(\omega t,x,\xi)$ and $H_0$ denotes the Weyl operator $(h_0)^{w}(x,D)$ with the symbol (\ref{h01'}).

\begin{rem}\label{remark}
Let $f$ be a polynomial of degree at most 2, then for the Weyl operator $h^{w}(x,D)$, the symbol of  $e^{\text{i}\kappa f^{w}(x,D)}h^{w}(x,D)e^{-\text{i}\kappa f^{w}(x,D)}$ is $h\circ\textbf{X}_f^\kappa$, where $\textbf{X}_f^\kappa$ is the Hamiltonian flow governed by $f$.
\end{rem}

Since the perturbation  $\epsilon W(\omega t,x,-\text{i}\partial_x)$ in (\ref{H}) is time dependent, it is useful to know how a time dependent mapping transforms a classical and a quantum Hamiltonian.

Let us consider a one-parameter family of Hamiltonian functions $f(t,x,\xi)$ in the phase space  $\mathbb{R}^{2d}$ with the symplectic form $\text{d}x\wedge\text{d}\xi$ (here $t$ is taken as an external parameter), and denote by $\textbf{X}_f^{\kappa}$ the Hamiltonian flow it generates. Consider the Hamilton's equations
\begin{equation*}
  \frac{\text{d}x}{\text{d}\kappa}=-\frac{\partial f}{\partial\xi}(t,x,\xi),\ \ \ \ \  \frac{\text{d}\xi}{\text{d}\kappa}=\frac{\partial f}{\partial x}(t,x,\xi)
\end{equation*}
and the time dependent coordinate transformation
\begin{equation}\label{coordinate}
  (x',\xi')=\phi(t,x,\xi):=\textbf{X}_f^{\kappa}(t,x,\xi)|_{\kappa=1}.
\end{equation}
After the coordinate transformation (\ref{coordinate}), a Hamiltonian system with Hamiltonian $h$ is changed into a new Hamiltonian system with Hamiltonian
\begin{equation}\label{coordinate1}
  h'(t,x,\xi):=h\circ\phi(t,x,\xi)-\int_{0}^1\frac{\partial f}{\partial t}(t,\textbf{X}_f^{\kappa}(t,x,\xi))\text{d}\kappa.
\end{equation}

For a Weyl operator $f^{w}$ related to the symbol $f$, \cite{Bambusi2018reducibility} gives the following property.

\begin{lem}[Remark 2.6 in \cite{Bambusi2018reducibility}]\label{operator}
If the operator $f^{w}(t,x,D)$ is self-adjoint for any fixed $t$, then the transformation
\begin{equation}\label{trans}
  \psi=e^{-\emph{i}f^{w}(t,x,D)}\psi'
\end{equation}
transform $\emph{i}\dot{\psi}=H\psi$ into $\emph{i}\dot{\psi}'=H'\psi'$ with
\begin{equation*}
  H'=e^{\emph{i}f^{w}(t,x,D)}He^{-\emph{i}f^{w}(t,x,D)} -\int_0^1e^{\emph{i}\kappa f^{w}(t,x,D)}(\partial_t f^{w}(t,x,D))e^{-\emph{i}\kappa f^{w}(t,x,D)}\emph{d}\kappa.
\end{equation*}
\end{lem}

Moveover,  the following properties hold.
\begin{lem}\label{prop0}
Let $f(\rho,x,\xi)$ be a polynomial in $(x,\xi)$ of degree at most 2 with real coefficients belonging to the class $C^{\ell}$ with respect to  $\rho\in\mathbb{R}^n$. Then $\forall \rho\in \mathbb{R}^n$, the operator $f^{w}(\rho,x,D)$ is self-adjoint in $L^2(\mathbb{R}^d)$. Furthermore, $\forall r\geq 0,\
 \forall \kappa\in\mathbb{R}$ the following holds true:
 \begin{enumerate}
   \item[(a)] the map $\rho\mapsto e^{-\emph{i}\kappa f^{w}(\rho,x,D)}\in C^0(\mathbb{R}^n,\mathcal{L}(\mathcal{H}^{r+2},\mathcal{H}^{r}))$;
   \item[(b)] $\forall \psi\in\mathcal{H}^r$, the map $\rho\mapsto e^{-\emph{i}\kappa f^{w}(\rho,x,D)}\psi\in C^0(\mathbb{R}^n,\mathcal{H}^r)$;
   \item[(c)] $\forall l\in\mathbb{N}$ satisfying $l\leq\ell$, the map $\rho\mapsto e^{-\emph{i}\kappa f^{w}(\rho,x,D)}\in C^l(\mathbb{R}^n,\mathcal{L}(\mathcal{H}^{r+4l+2},\mathcal{H}^{r}))$;
   \item[(d)] if the coefficients of $f(\rho,x,\xi)$ are uniformly bounded in $\rho\in\mathbb{R}^n$, then for any $r>0$ there exists $c_r>0,\ C_r>0$ such that
       \begin{equation*}
         c_r\|\psi\|_r\leq\|e^{-\emph{i}\kappa f^{w}(\rho,x,D)}\psi\|_r\leq C_r\|\psi\|_r,\ \ \ \  \forall \rho\in\mathbb{R}^n, \ \forall \kappa\in[0,1].
       \end{equation*}
 \end{enumerate}
\end{lem}

\begin{proof}
The proof of this lemma is similar to that of Lemma 2.8 in \cite{Bambusi2018reducibility}, so the proof is omitted here.
\end{proof}

Remark \ref{remark}, (\ref{coordinate1}), Lemmata \ref{operator} and \ref{prop0} imply the following important proposition.
\begin{prop}\label{prop1}
Let $f(t,x,\xi)$ be a polynomial of degree at most 2 in $(x,\xi)$ with  time dependent real coefficients belonging to class $C^{\ell}$. If the transformation (\ref{coordinate}) transforms a classical system with Hamiltonian $h$ into a Hamiltonian system with Hamiltonian $h'$, then the transformation (\ref{trans}) transforms the quantum system with Hamiltonian $h^{w}$ into the quantum system with Hamiltonian $(h')^{w}$.
\end{prop}

As a consequence, for quadratic Hamiltonians, the quantum KAM theorem  follows from the corresponding classical KAM theorem.

In order to obtain Theorem \ref{main}, let us consider the  time-dependent differentiable Hamiltonian
\begin{equation}\label{hepsilon}
h_{\epsilon}(\omega t,x,\xi):=h_0''(x,\xi)+\epsilon W(\omega t,x,\xi),\ \ \ \ \ h_0''(x,\xi):=\sum_{j=1}^d v_j\frac{x_j^2+\xi_j^2}{2},
\end{equation}
where $v_j,\ j=1,\cdots,n$ are defined in condition \textbf{C1} and $W$ satisfies condition \textbf{C3}. Without loss of generality, let $\epsilon>0$. The following KAM theorem holds.

\begin{thm}\label{kam}
Assume that the differentiable Hamiltonian $h_{\epsilon}(\omega t,x,\xi)$ defined in (\ref{hepsilon}) satisfies conditions \textbf{C1} and \textbf{C3}.
Then for any $0<\gamma=\gamma(n, d, v_1,\cdots, v_d)\ll 1$, there exists $0<\epsilon_{\star}\ll \gamma^{8}$  such that for $0<\epsilon<\epsilon_{\star}$, the following holds true:
\begin{enumerate}
  \item[(a)] there exists a closed set $\mathcal{D}_{\epsilon}\subset (0,2\pi)^n$ with  $\meas((0,2\pi)^n\setminus\mathcal{D}_{\epsilon})\leq 4\gamma^{\frac{1}{2}}$;
  \item[(b)] for any $\omega\in\mathcal{D}_{\epsilon}$, there exists an analytic map $\theta\ni\mathbb{T}^n\mapsto A_{\omega}(\theta)\in \emph{sp}(2d)$ \footnote{sp(2d) is the symplectic algebra composed of all $2d\times 2d$ Hamiltonian matries.}
   and an analytic map $\theta\ni\mathbb{T}^n\mapsto V_{\omega}(\theta)\in \mathbb{R}^{2d}$, such that the  coordinate transformation
      \begin{equation}\label{coordinate2}
      (x',\xi')=e^{A_{\omega}(\omega t)}(x,\xi)+V_{\omega}(\omega t)
      \end{equation}
      where $\sup_{(\theta,\omega)\in\mathbb{T}^n\times\mathcal{D}_{\epsilon}}\|A_{\omega}\|,\ \sup_{(\theta,\omega)\in\mathbb{T}^n\times\mathcal{D}_{\epsilon}}\|V_{\omega}\|\leq \epsilon^{\frac{1}{4}}$, conjugates the Hamilton's equations of (\ref{hepsilon}) to the Hamilton's equations of
       \begin{equation}\label{hinfty}
         h_{\infty}(x,\xi)=e_{\infty}+\sum_{j=1}^d v_j^{\infty}\frac{x_j^2+\xi_j^2}{2}
       \end{equation}
where $ e_{\infty}=e_{\infty}(\omega)$ and $v_j^{\infty}=v_j^{\infty}(\omega)$ are defined on $\mathcal{D_{\epsilon}^{\star}}$ and fulfill the estimates
\begin{equation}\label{v}
   |e_{\infty}|, \ |v_j^{\infty}-v_j|\leq \epsilon^{\frac{1}{2}},\ \ \ \ \ j=1,\cdots,d.
\end{equation}
\end{enumerate}

\end{thm}

The proof of Theorem \ref{kam} is given in Section \ref{proofkam}. Theorem \ref{main} follows immediately from Proposition \ref{prop1} and Theorem \ref{kam}. The proofs of Corollary \ref{hsmorm} and Corollary \ref{boundstate} are similar to that of \cite{Bambusi2018reducibility}, so they are  omitted here.

\begin{proof}[Proof of  Theorem \ref{main}.]

It is easily to see that the  coordinate transformation (\ref{coordinate2}) has the form (\ref{coordinate}) with a Hamiltonian $f_{\omega}(\omega t,x,\xi)$ which is a polynomial in $(x,\xi)$ of degree at most 2 with real, differentiable ($\ell$ times) in $\theta\in\mathbb{T}^n$ and uniformly bounded coefficients in $t\in\mathbb{R}$.

Define $U_{\omega}(\omega t)=e^{-\text{i}f_{\omega}^{w}(\omega t,x,D)}$. By Proposition \ref{prop1}, the transformation $\psi=U_{\omega}(\omega t)\psi'$ transforms the original equation (\ref{eq}) into $\text{i}\dot{\psi}'=H_{\infty}\psi'$, where
\begin{equation*}
 H_{\infty}= H_{\infty}(x,-\text{i}\nabla,\omega)=\sum_{j=1}^d v_j^{\infty}\left(x_j^2-\left(\frac{\partial}{\partial x_j}\right)^2\right)
\end{equation*}
is the Weyl operator of $h_{\infty}$ defined in (\ref{hinfty}) and  $H_{\infty}$ is independent of $t$.

Furthermore, $\theta\mapsto U_{\omega}(\theta)$ fulfills (a)-(d) of Lemma \ref{prop0}, and therefore $\theta\mapsto U_{\omega}(\theta)$ fulfills items (1) and (3) of Theorem \ref{main}.  Concerning item (2) and  by Taylor's formula, the quantity
$\|U_{\omega}(\theta)-id\|_{\mathcal{L}(\mathcal{H}^{r+2},\mathcal{H}^r)}$ is controlled by $\|f^w_{\omega}(\theta,x,D)\|_{\mathcal{L}(\mathcal{H}^{r+2},\mathcal{H}^r)}$, from which estimation (\ref{u}) follows.
The proof of Theorem \ref{main} is finished.
\end{proof}

\section{A differentiable KAM result: the proof of Theorem \ref{kam} }\label{proofkam}

In this section, the proof of Theorem \ref{kam} is given. This is based on analytic approximation lemma and KAM iteration techniques.

\subsection{Analytic approximation}\label{appro}

An analytical approximation lemma is given in this subsection in order to find a series of matrices and vectors which are analytic in some complex strip domains to approximate the perturbation $\epsilon W(\theta,x,\xi)$ defined in (\ref{hepsilon}).

To start with, some definitions and notations are given. Assume that $X$ is a Banach
space with the norm $\|\cdot\|_{X}$. Recall that $C^{\mu}(\mathbb{R}^n;X)$ for $0\leq\mu<1$ denotes the space of bounded
H\"older continuous functions $g:\mathbb{R}^n\mapsto X$ with the norm
\begin{equation*}
\|g\|_{C^{\mu},X}=\sup_{0<|x-y|<1}\frac{\|g(x)-g(y)\|_{X}}{|x-y|^{\mu}}+\sup_{x\in\mathbb{R}^n}\|g(x)\|_{X}.
\end{equation*}
If $\mu=0$, then $\|g\|_{C^{\mu},X}$ denotes the sup-norm. For $\ell=l+\mu$ with $l\in\mathbb{N}$ and $0\leq\mu<1$, $C^{\ell}(\mathbb{R}^n;X)$ is denoted as the space of functions $g:\mathbb{R}^n\mapsto X$ with H\"older continuous partial derivatives, i.e., $\partial^{\alpha}g\in C^{\mu}(\mathbb{R}^n;X^{\alpha})$ for all multi-indices $\alpha=(\alpha_1,\cdots,\alpha_n)\in\mathbb{N}^n$ with the assumption that $|\alpha|:=|\alpha_1|+\cdots+|\alpha_n|\leq l$ and $X^{\alpha}$ is the Banach space of bounded operators $T:\prod^{|\alpha|}(\mathbb{R}^n)\rightarrow X$ with the norm
\begin{equation}\label{t}
  \|T\|_{X^{\alpha}}=\sup\{\|T(u_1,\cdots,u_{|\alpha|})\|_{X}:\|u_i\|=1,1\leq i\leq |\alpha|\}.
\end{equation}
Define the norm
\begin{equation*}
  \|g\|_{C^{\ell}(X)}=\sup_{|\alpha|\leq \ell}\|\partial^{\alpha}g\|_{C^{\mu},X^{\alpha}}.
\end{equation*}

\begin{lem}[Jackson, Moser, Zehnder]\label{smooth}
Let $g\in C^{\ell}(\mathbb{R}^n;X)$ for some $\ell>0$ with finite $C^{\ell}$ norm over $\mathbb{R}^n$. Let $\varphi$ be a radical-symmetric, $C^{\infty}$  function, having as support the closure of the unit ball centered at the origin, where $\varphi$ is completely flat and takes value 1. Let $\psi=\hat{\varphi}$ be the Fourier transform of $\varphi$. For all $\sigma>0$, define
\begin{equation*}
  g_{\sigma}(x)=\psi_{\sigma}\ast g=\frac{1}{\sigma^n}\int_{\mathbb{R}^n} \psi(\frac{x-y}{\sigma})g(y)\text{d}y.
\end{equation*}
Then there exists a constant $C\geq 1$ depending only on $\ell$ and $n$ such that the following holds: For
any $\sigma>0$, the function $g_{\sigma}$ is a real-analytic function from $\mathbb{C}^n$ to $X$ such that if $\bigtriangleup_{\sigma}^{n}$ denotes the
$n$-dimensional complex strip of width $\sigma$,
\begin{equation*}
  \bigtriangleup_{\sigma}^{n}=\{x\in\mathbb{C}^n:\max_{1\leq j\leq n}|\Im x_j|\leq\sigma\},
\end{equation*}
 then for any $\alpha\in\mathbb{N}^{n}$ with $|\alpha|\leq \ell$, one has
 \begin{equation*}
   \sup_{x\in \bigtriangleup_{\sigma}^{n}} \|\partial^{\alpha}g_{\sigma}(x) -\sum_{|\beta|\leq \ell-|\alpha|}\frac{\partial^{\beta+\alpha}g(\Re x)}{\beta!}(\emph{i}\Im x)^{\beta}\|_{X^{\alpha}} \leq C\|g\|_{C^{\ell}}\sigma^{\ell-|\alpha|},
 \end{equation*}
where $\beta!:=\beta_1!\cdots\beta_n!$ and $y^{\beta}:=y_1^{\beta_1}\cdots y_n^{\beta^n}$ for $y=(y_1,\cdots,y_n)\in \mathbb{C}^n$,
and for all $0\leq s\leq \sigma$,
 \begin{equation*}
   \sup_{x\in \bigtriangleup_{s}^{n}} \|\partial^{\alpha}g_{\sigma}(x) -\partial^{\alpha}g_{s}(x)\|_{X^{\alpha}} \leq C\|g\|_{C^{\ell}}\sigma^{\ell-|\alpha|}.
 \end{equation*}
The function $g_{\sigma}$ preserves periodicity (i.e., if $g$ is $T$-periodic in any of its variables $x_j$, so is $g_{\sigma}$). Finally, if $g$ depends on some parameters $\omega\in \Pi\subset \mathbb{R}^n$ and if the Lipschitz-norm of $g$
\begin{equation*}
  \|g(x,\omega)\|_{C^{\ell}}^{\mathcal{L}}:=\|g(x,\omega)\|_{C^{\ell}} + \left(\sum_{j=1}^n\|\partial_{\omega_j}g(x,\omega)\|^2_{C^{\ell}}\right)^{\frac{1}{2}}
\end{equation*}
and its $x$-derivatives are uniformly bounded by $\|g\|_{C^{\ell}}^{\mathcal{L}}$, then all the above estimates hold with $\|\cdot\|$ replaced by $\|\cdot\|^{\mathcal{L}}$.
\end{lem}

This lemma is very similar with the approximation theory obtained by Jackson, Moser and Zehnder, the only difference is that we extend the applied range from $C(\mathbb{R}^n;\mathbb{C}^n)$ to $C(\mathbb{R}^n;X)$. The proof of this lemma consists in a direct check which is based on standard tools from calculus and complex analysis, for details see \cite{Salamon1989KAM}, \cite{Salamon2004The} and references therein. For ease of notation, $\|\cdot\|_{X}$ shall be replaced by $\|\cdot\|$.

Now let us apply this lemma to a $C^{\ell}$ real-valued function $P(\theta)$. Fix a sequence of fast decreasing numbers $\sigma_m\downarrow 0,\ m\geq0$, and $\sigma_0\leq\frac{1}{4}$. Construct a sequence of real analytic functions $P^{(m)}(\theta)$ such that the following conclusions hold:
\begin{enumerate}
  \item $P^{(m)}(\theta)$ is real analytic on the complex strip $\mathbb{T}^n_{2\sigma_m}=\{\theta\in\mathbb{C}^n/(2\pi\mathbb{Z})^n:|\Im\theta|\leq 2\sigma_m\}$, where $|\Im\theta|=\max_{1\leq j\leq n}|\Im\theta_j|$.
  \item The sequence of functions $P^{(m)}(\theta)$ satisfies the bounds:
  \begin{equation}\label{pv}
    \sup_{\theta\in\mathbb{T}^n}\|P^{(m)}(\theta)-P(\theta)\|\leq C\|P\|_{C^{\ell}}\sigma_m^{\ell},
  \end{equation}
  \begin{equation*}
     \sup_{\theta\in\mathbb{T}^n_{2\sigma_{m+1}}}\|P^{(m+1)}(\theta)-P^{(m)}(\theta)\|\leq C\|P\|_{C^{\ell}}\sigma_m^{\ell},
  \end{equation*}
  where $C$ denotes a constant depending only on $n$ and $\ell$.
  \item The first approximate $P^{(0)}$ is ``small" with the perturbation $P$. Precisely speaking, for arbitrary $\theta\in\mathbb{T}^n_{2\sigma_0}$, it follows that
  \begin{align*}
    \|P^{(0)}(\theta)\| &\leq \|P^{(0)}(\theta)-\sum_{|\alpha|\leq \ell}\frac{\partial^{\alpha}P(\Re\theta)}{\alpha!}(\text{i}\Im \theta)^{\alpha}\| +\|\sum_{|\alpha|\leq \ell}\frac{\partial^{\alpha}P(\Re\theta)}{\alpha!}(\text{i}\Im \theta)^{\alpha}\| \\
     & \leq C\|P\|_{C^{\ell}}\sigma_0^{\ell}+\sum_{0\leq m\leq\ell}\|P\|_{C^{m}}\sigma_0^{m}\\
     &\leq C\|P\|_{C^{\ell}}\sum_{m=0}^{\ell}\sigma_0^m\\
     &\leq C\|P\|_{C^{\ell}},
  \end{align*}
  where constant $C$ is independent of $\sigma_0$, and the last inequality holds true due to the hypothesis that $\sigma_0\leq\frac{1}{4}$.
  \item  By denoting $P_0=P^{(0)}$, $P_m=P^{(m)}-P^{(m-1)},\ m=1,2,\cdots,$ and  from (\ref{pv}), it appears
  \begin{equation}\label{sum}
    P(\theta)=P^{(0)}(\theta)+\sum_{m=0}^{\infty}\left(P^{(m+1)}(\theta)-P^{(m)}(\theta)\right)=\sum_{m=0}^{\infty}P_m(\theta),\ \ \ \ \theta\in\mathbb{T}^n,
  \end{equation}
  where $P_m(\theta)$ are real analytic in $\mathbb{T}_{2\sigma_m}^n,\ m=1,2,\cdots$ with the estimations
  \begin{equation}\label{pvphi}
  \sup_{\theta\in\mathbb{T}^n_{2\sigma_0}}\|P_0(\theta)\|\leq C\|P\|_{C^{\ell}},\ \ \ \ \  \sup_{\theta\in\mathbb{T}^n_{2\sigma_m}}\|P_m(\theta)\|\leq C\|P\|_{C^{\ell}}\sigma_{m-1}^{\ell},\ \ \ \  m=1,2,\cdots.
  \end{equation}
\end{enumerate}

Consider the perturbation $W_0=\epsilon W$ defined in (\ref{hepsilon}), which is a  polynomial  of degree at most 2 in $(x, \xi)\in\mathbb{R}^{2d}$ with coefficients belonging to the class $C^{\ell}$ in $\theta\in\mathbb{T}^n$ and  $\mathbb{C}^1$ in  $\omega\in\mathcal{D}$ with the form
\begin{equation}\label{w}
  W_0(\theta,x,\xi)=\langle x,W_0^{xx}(\theta)x\rangle +\langle x,W_0^{x\xi}(\theta)\xi\rangle+\langle \xi,W_0^{\xi\xi}(\theta)\xi\rangle+\langle W_0^{x}(\theta),x\rangle+\langle W_0^{\xi}(\theta),\xi\rangle+W_0^{\theta}(\theta),
\end{equation}
where for any $\theta\in\mathbb{T}^n$,  $W_0^{xx}(\theta), W_0^{\xi\xi}(\theta),W_0^{x\xi}(\theta)$ are real $d\times d$-matrices, $W_0^{x}(\theta),W_0^{\xi}(\theta)$ are vectors in $\mathbb{R}^{d}$, $W_0^{\theta}(\theta)$ is real number.
Assume
\begin{equation*}
\sup_{\omega\in \mathcal{D}}\|W_0^{\varrho\varrho'}\|_{C^{\ell}},\
\sup_{\omega\in \mathcal{D}}\|\partial_{\omega_j}W_0^{\varrho\varrho'}\|_{C^{\ell}}\leq \epsilon, \ \ \ \ \varrho\varrho'\in\{xx,x\xi,\xi\xi\},\ \ \ \ j=1,\cdots,n,
\end{equation*}
\begin{equation*}
  \sup_{\omega\in \mathcal{D}}\|W_0^{\varrho}\|_{C^{\ell}},\ \sup_{\omega\in \mathcal{D}}\|\partial_{\omega_j}W_0^{\varrho}\|_{C^{\ell}}\leq \epsilon,\ \  \ \  \varrho\in\{x,\xi,\theta\},\ \ \ \  j=1,\cdots,n,
\end{equation*}
and let
\begin{equation}\label{sv}
 \sigma_{m}=2\epsilon_{m+1}^{\frac{1}{\ell}},\ \ \ \  \ \epsilon_m=\epsilon_0^{\left(1+\rho\right)^{m}}, \ \ \ \ m=0,1,2,\cdots.
\end{equation}
Then the sequence of real numbers $\{\sigma_m\}_{m=0}^{\infty}$  goes fast to zero, and for $\omega\in\mathcal{D}$,  there are approximation sequences of $W_0^{\varrho\varrho'}, W_0^{\varrho} $  having the forms
\begin{equation*}
 W_0^{\varrho\varrho'}(\theta)=\sum_{m=0}^{\infty}\epsilon_mW^{\varrho\varrho'}_{0,m}(\theta), \ \ \ \ \ \
 W_0^{\varrho}(\theta)=\sum_{m=0}^{\infty}\epsilon_mW^{\varrho}_{0,m}(\theta),
\end{equation*}
where $W^{\varrho\varrho'}_{0,m}(\theta),\  W^{\varrho}_{0,m}(\theta)$ are real analytic on $\mathbb{T}^n_{2\sigma_m}$ with
\begin{equation*}
  \|W^{\varrho\varrho'}_{0,m}\|_{2\sigma_m}=\sup_{\theta\in\mathbb{T}^n_{2\sigma_m}}\|W^{\varrho\varrho'}_{0,m}(\theta)\|\leq C,\ \ \ \ \ \|W^{\varrho}_{0,m}\|_{2\sigma_m}=\sup_{\theta\in\mathbb{T}^n_{2\sigma_m}}\|W^{\varrho}_{0,m}(\theta)\| \leq C,\ \ \ \  \ m=0,1,\cdots.
\end{equation*}
Here $C$ denotes different constants depending only on $n$ and $\ell$ and $\|\cdot\|$ demotes the Euclidean norm of vectors in $\mathbb{C}^d$ or $1$-norm of $d\times d$  complex matrices.
Denote
\begin{equation}\label{wj}
W_{0,m}(\theta,x,\xi)=\langle x,W_{0,m}^{xx}(\theta)x\rangle+\langle x,W_{0,m}^{x\xi}(\theta)\xi\rangle+\langle \xi,W_{0,m}^{\xi\xi }(\theta)\xi\rangle
                      +\langle W_{0,m}^{x}(\theta ),x\rangle+\langle W_{0,m}^{\xi}(\theta),\xi\rangle +W_{0,m}^{\theta}(\theta)
\end{equation}
for $ m=0,1,\cdots$. Then $W_{0}(\theta,x,\xi)=\sum_{m=0}^{\infty}W_{0,m}(\theta,x,\xi).$

As a matter of convenience,  introducing complex variables defined by
\begin{equation*}
z_j=\frac{\xi_j-\text{i}x_j}{\sqrt{2}}, \ \ \ \ \  \bar{z}_j=\frac{\xi_j+\text{i}x_j}{\sqrt{2}},\ \ \ \ \ j=1,\cdots,d.
\end{equation*}
Here $z_j$ and $\bar{z}_j$ are regarded as independent variables.  The symplectic structure of $\mathbb{C}^{2d}$ is  $\text{id}z\wedge\text{d}\bar{z}$.
In this framework, $h''_0$ defined in (\ref{hepsilon}) has an extension with the form
\begin{equation}\label{h}
h''_0(x,\xi)=\sum_{j=1}^d v_jz_j\bar{z}_j=\langle z,N_0\bar{z}\rangle,\ \ \ \  N_0=\text{diag}(v_j:j=1,\cdots,n).
\end{equation}
Let $q_{0,m}(\theta,z,\bar z)=W_{0,m}(\theta,x,\xi), m=0,1,\cdots$ with the form
\begin{equation}\label{qv}
q_{0,m}(\theta,z,\bar{z})=\langle z,Q_{0,m}^{zz}(\theta)z\rangle+\langle z,Q_{0,m}^{z\bar z}(\theta)\bar z\rangle+\langle \bar z,{Q}_{0,m}^{\bar z\bar z}(\theta)\bar{z}\rangle
                      +\langle Q_{0,m}^{z}(\theta ),z\rangle+\langle {Q}_{0,m}^{\bar z}(\theta),\bar{z}\rangle +Q^{\theta}_{0,m}(\theta),
\end{equation}
where
\begin{equation*}
  Q_{0,m}^{zz}=\frac{1}{2}(W_{0,m}^{\xi\xi}-W_{0,m}^{xx}+\text{i}W_{0,m}^{x\xi}),\ \  \ \ \ Q_{0,m}^{\bar{z}\bar{z}}=\frac{1}{2}(W_{0,m}^{\xi\xi}-W_{0,m}^{xx}-\text{i}W_{0,m}^{x\xi}),
\end{equation*}
\begin{equation*}
   Q_{0,m}^{z\bar{z}}=\frac{1}{2}(W_{0,m}^{\xi\xi}+(W_{0,m}^{\xi\xi})^{\top}+W_{0,m}^{xx}+(W_{0,m}^{xx})^{\top} +\text{i}W_{0,m}^{x\xi}-\text{i}(W_{0,m}^{x\xi})^{\top}),
\end{equation*}
\begin{equation*}
  Q_{0,m}^{z}=\frac{\sqrt{2}}{2}(W^{\xi}_{0,m}+\text{i}W^{x}_{0,m}),\ \ \ \ \ Q_{0,m}^{\bar{z}}=\frac{\sqrt{2}}{2}(W^{\xi}_{0,m}-\text{i}W^{x}_{0,m}),\ \ \ \ \  Q^{\theta}_{0,m}=W^{\theta}_{0,m}
\end{equation*}
are analytic with respect to $\theta\in\mathbb{T}^n_{\sigma_m}$ and $C^1$ in $\omega\in\mathcal{D}$
with
\begin{equation*}
  \|Q^{\varrho\varrho'}_{0,m}\|_{\sigma_m}\leq C,\ \ \ \ \  \|Q^{\varrho}_{0,m}\|_{\sigma_m} \leq C,\ \ \ \  \ m=0,1,\cdots.
\end{equation*}
Obviously, for $(\theta,\omega)\in\mathbb{T}^n\times\mathcal{D}$, $Q_{0,m}^{z\bar{z}}$ is Hermitian matrix, $Q_{0,m}^{\bar{z}\bar{z}}$ is the conjugate matrix of $Q_{0,m}^{{z}{z}}$, $Q_{0,m}^{\bar{z}}$ is the conjugate vector of $Q_{0,m}^{{z}}$ and $Q^{\theta}_{0,m}$ is real number.
Then $W_0=\epsilon W$  has an extension
\begin{equation}\label{q01}
q_0(\theta,z,\bar z)=q_0^{(0)}+\sum_{m=1}^{\infty}\left(q^{(m)}_0-q^{(m-1)}_0\right)=\sum_{m=0}^{\infty}\epsilon_mq_{0,m}(\theta,z,\bar z),
\end{equation}
where $q^{(m)}_0$ and $q_{0,m}$ are real-analytic in $\mathbb{T}^n_{\sigma_m}$. Moreover, $q_0$ is also a polynomial in $(z,\bar{z})$ of degree two with the form
\begin{equation}\label{q0}
q_0(\theta,z,\bar{z})=\langle z,Q_0^{zz}(\theta)z\rangle+\langle z,Q_0^{z\bar z}(\theta)\bar z\rangle+\langle \bar z,{Q}_0^{\bar{z}\bar{z}}(\theta)\bar{z}\rangle
                      +\langle Q_0^{z}(\theta ),z\rangle+\langle {Q}_0^{\bar{z}}(\theta),\bar{z}\rangle+Q_0^{\theta}(\theta),
\end{equation}
where $Q_0^{zz}$, $Q_0^{z\bar{z}}$, $Q_0^{\bar{z}\bar{z}}$, $Q_0^{z}$, $Q_0^{\bar z}$, $Q_0^{\theta}$ are all in the class $C^{\ell}$ with respect to $\theta\in \mathbb{T}^n$.

Suppose that a Hamiltonian $r$ defined in $\mathbb{T}^n_s\times\mathbb{C}^{2d}\times\mathcal{D}$ has the form
\begin{equation}\label{rhamiltonian}
  r(\theta,z,\bar{z})=\langle z,R^{zz}(\theta)z\rangle+\langle z,R^{z\bar z}(\theta)\bar z\rangle+\langle \bar z,R^{\bar{z}\bar{z}}(\theta)\bar{z}\rangle
                      +\langle R^{z}(\theta ),z\rangle+\langle R^{\bar{z}}(\theta),\bar{z}\rangle+R^{\theta}(\theta),
\end{equation}
 and satisfies the following three conditions
\begin{itemize}
  \item[(\text{a1})] $r(\theta,z,\bar{z};\omega)$ is analytic in $\theta\in\mathbb{T}^n_{s}$, $C^1$ in $\omega\in\mathcal{D}$,
  \item[(\text{a2})] for all $(\theta,\omega)\in\mathbb{T}^n\times\mathcal{D}$, $R^{z\bar{z}}$ are Hermitian matrices, $R^{\bar{z}\bar{z}}$ are the conjugate matrices of $R^{{z}{z}}$, $R^{\bar{z}}$ are the conjugate vectors of $R^{{z}}$,
  \item[(\text{a3})] $r(\theta,z,\bar{z};\omega)$ has the finite norm:
  \begin{equation}\label{normm}
  [r]_{s,\mathcal{D}}=[r]^{(0)}_{s,\mathcal{D}}+[r]^{(1)}_{s,\mathcal{D}}<\infty,
\end{equation}
\begin{equation*}
  [r]^{(0)}_{s,\mathcal{D}}=\|R^{zz}\|_{s,\mathcal{D}} +\|R^{z\bar{z}}\|_{s,\mathcal{D}} +\|R^{\bar{z}\bar{z}}\|_{s,\mathcal{D}}+\|R^{z}\|_{s,\mathcal{D}}+\|R^{\bar z}\|_{s,\mathcal{D}}+\|R^{\theta}\|_{s,\mathcal{D}},
\end{equation*}
\begin{align*}
  [r]^{(1)}_{s,\mathcal{D}}=&\left(\sum_{j=1}^n\|\partial_{\omega_j}R^{zz}\|^2_{s,\mathcal{D}}\right)^{\frac{1}{2}} +\left(\sum_{j=1}^n\|\partial_{\omega_j}R^{z\bar{z}}\|^2_{s,\mathcal{D}}\right)^{\frac{1}{2}}
  +\left(\sum_{j=1}^n\|\partial_{\omega_j}R^{\bar{z}\bar{z}}\|^2_{s,\mathcal{D}}\right)^{\frac{1}{2}}\\
  &\ +\left(\sum_{j=1}^n\|\partial_{\omega_j}R^{z}\|^2_{s,\mathcal{D}}\right)^{\frac{1}{2}}
  +\left(\sum_{j=1}^n\|\partial_{\omega_j}R^{\bar{z}}\|^2_{s,\mathcal{D}}\right)^{\frac{1}{2}}
  +\left(\sum_{j=1}^n\|\partial_{\omega_j}R^{\theta}\|^2_{s,\mathcal{D}}\right)^{\frac{1}{2}},
\end{align*}
where $\|A\|_{s,\mathcal{D}}=\sup_{(\theta,\omega)\in\mathbb{T}^n_s\times\mathcal{D}}\|A(\theta,\omega)\|$ with $\|\cdot\|$ denoting the 1-norm for matrices or Euclidean norm for vectors.
\end{itemize}
Denote
\begin{equation}\label{qsd}
 \mathcal{Q}_{s,\mathcal{D}}=\{r:r \text{ has the form (\ref{rhamiltonian}) and  satisfies (a1),(a2) and (a3)}\},
\end{equation}
\begin{equation}\label{qsdr}
 \mathcal{Q}^{\Re}_{s,\mathcal{D}}=\{r\in \mathcal{Q}_{s,\mathcal{D}}:R^{\theta}(\theta)\in\mathbb{R},\ \forall (\theta,\omega)\in\mathbb{T}^n\times\mathcal{D} \},
\end{equation}
\begin{equation}\label{qsdi}
 \mathcal{Q}^{\Im}_{s,\mathcal{D}}=\{r\in \mathcal{Q}_{s,\mathcal{D}}:R^{\theta}(\theta)\in\text{i} \mathbb{R},\ \forall (\theta,\omega)\in\mathbb{T}^n\times\mathcal{D} \}.
\end{equation}

 It is clear that  $q_{0,m}\in \mathcal{Q}^{\Re}_{\sigma_m,\mathcal{D}},\ m=0,1,\cdots$ with the estimates
\begin{equation}
[q_{0,m}]_{\sigma_m,\mathcal{D}}\leq C,\ m=0,1,\cdots,
\end{equation}
where $C$ are different constants depending only on $n$ and $\ell$.

Now the Hamiltonian system associated with the time dependent Hamiltonian $h''_0(x,\xi)+\epsilon W(\omega t,x,\xi)$ is then equivalent to the Hamiltonian system $h_0(z,\bar{z})+ q_0(\theta,z,\bar z)$, where $q_0$ is defined by (\ref{q01}) and
\begin{equation}\label{h0}
h_0(z,\bar{z})=\langle z,N_0\bar z\rangle.
\end{equation}

\subsection{General strategy}
Consider  a time dependent Hamiltonian  $h_0$ defined by (\ref{h0}) being in normal form  and a small perturbation $q_0(\theta,z,\bar{z})$ defined by (\ref{q01}) in  $\mathbb{T}^n\times\mathbb{C}^{2d}$.

The purpose is to construct, inductively, real-analytic  transformations $\widetilde{\Phi}_m, m=0,1,\cdots$, so that
\begin{equation}\label{m'}
  (h_0+q^{(m)}_{0})\circ \widetilde{\Phi}_m=h_{m+1}+\epsilon_{m+1}q_{m+1},
\end{equation}
where $h_{m+1}=\langle z,N_{m+1}\bar z\rangle+e_{m+1}$ is in ``normal form", $\epsilon_m$-close to $h_m$, and the sequence of real-analytic functions $\epsilon_mq_m$'s are  perturbations of smaller and smaller size $\mathcal{O}(\epsilon_m)$.
The map $\widetilde{\Phi}_m$ is seeked of the form
$$\widetilde{\Phi}_m=\Phi_0\circ\Phi_1\circ\cdots\circ\Phi_m,$$
where $\Phi_m$ are  time-one flows of some Hamiltonians $\epsilon_mf_m$: $\Phi_m=\mathbf{X}_{\epsilon_mf_m}^\kappa|_{\kappa=1}$.
Thus, by induction for $m=0,1,\cdots$, (\ref{m'}) takes the form
\begin{equation}\label{m''}
(h_{m}+\epsilon_mq_m+\epsilon_mq_{0,m}\circ\widetilde{\Phi}_{m-1})\circ\Phi_m=h_{m+1}+\epsilon_{m+1}q_{m+1}.
\end{equation}
By choosing suitable $\sigma_m$, the term $q_{0,m}$ can be controlled by $q_m$. Meanwhile,  a series of closed sets $\mathcal{D}_{m+1}\subset\mathcal{D}_{m}$ and  Hamiltonians $\epsilon_mf_{m}\in\mathcal{Q}^{\Im}_{s_{m+1},\mathcal{D}_{m+1}}$($s_m<\sigma_m$) may be found. Therefore, (\ref{m'}) may be rewritten as
\begin{equation}\label{m}
(h_m+\epsilon_mq'_{m})\circ \Phi_{m}=h_{m+1}+\epsilon_{m+1}q_{m+1},\ \ \omega\in \mathcal{D}_{m+1},
\end{equation}
with
\begin{equation}\label{m'''}
  q'_0=q_{0,0},\ \ \ \ \ q_m'=q_m+q_{0,m}\circ\widetilde{\Phi}_{m-1},\ \ m=1,2,\cdots.
\end{equation}
Thus, $\epsilon_mq'_{m}$ and  $\epsilon_mq_m$ are both of size $\mathcal{O}(\epsilon_m)$, and (\ref{m}) fits now in more standard KAM approaches.
Notice that in the $(m+1)$-th step,  $h_m$ and  $q'_{m}$  are defined on $\mathcal{D}_m$ but the equality (\ref{m}) holds only on $\mathcal{D}_{m+1}$, from which the ``resonant parts" is excised.

A remark is that in order for this approach to work, the map $\widetilde{\Phi}_m$ has to verify suitable compatibility relations with respect to the analyticity domains (compare the inductive relation (\ref{m''})). More precisely, for $\omega\in\mathcal{D}_m$, if the  analyticity domain of $q_m$ is $\mathbb{T}^n_{s_m}\times \mathbb{C}^{2d}$, one has to show that
\begin{equation}\label{maprelation}
  \Phi_m: \mathbb{T}^n_{s_{m+1}}\times \mathbb{C}^{2d}\rightarrow\mathbb{T}^n_{s_m}\times \mathbb{C}^{2d},\ \ \ \
  \widetilde{\Phi}_{m-1}: \mathbb{T}^n_{s_{m}}\times \mathbb{C}^{2d}\rightarrow\mathbb{T}^n_{\sigma_m}\times \mathbb{C}^{2d},\ \  m=1,2,\cdots.
\end{equation}
To this end, choose $s_m=\frac{1}{2}\sigma_m,\ m=0,1,\cdots$.

As a consequence of the Hamiltonian structure,
\begin{align*}
h_{m+1}+\epsilon_{m+1}q_{m+1}=&\ (h_m+\epsilon_mq'_{m})\circ \Phi_{m}\\
=&\  h_m+ \epsilon_mq'_{m}+\epsilon_m\{h_m,f_m\}-\epsilon_m\omega\cdot\partial_{\theta}f_m\\
& \  +\epsilon_m^2\int_0^1(1-\kappa)\left\{\{h_m,f_m\}-\omega\cdot\partial_{\theta}f_m,f_m\right\}\circ\textbf{X}_{\epsilon_mf_m}^{\kappa}\text{d}\kappa \\
& \  +\epsilon_m^2\int_0^1\{q'_{m},f_m\}\circ\textbf{X}_{\epsilon_mf_m}^{\kappa}\text{d}\kappa,
\end{align*}
where  the Poisson bracket of  $f(z,\bar z)$ and $g(z,\bar z)$ is defined as
\begin{equation}\label{possion}
  \{f,g\}:= -\text{i}\frac{\partial f}{\partial z}\cdot\frac{\partial g}{\partial\bar{z}}+\text{i}\frac{\partial f}{\partial\bar{z}}\cdot\frac{\partial g}{\partial z}.
\end{equation}
The ``linearized equation" to be solved for $f_m$ has the form
\begin{equation*}
 h_m+ \epsilon_mq'_{m}+\epsilon_m\{h_m,f_m\}-\epsilon_m\omega\cdot\partial_{\theta}f_m=h_{m+1}+\mathcal{O}(\epsilon_{m+1}),\ \ \ \ \omega\in \mathcal{D}_{m+1},
\end{equation*}
or equivalently
\begin{equation}\label{interation}
  \epsilon_m q'_{m}+\epsilon_m\{h_m,f_m\}-\epsilon_m\omega\cdot\partial_{\theta}f_m=\langle z,\widetilde{N}_{m+1}\bar z\rangle + \tilde{e}_{m+1}+\mathcal{O}(\epsilon_{m+1}),\ \ \ \ \omega\in \mathcal{D}_{m+1}.
\end{equation}
Here, $e_1=\tilde{e}_1,\ e_{m+1}=e_m+\tilde{e}_{m+1},m=1,\cdots $ are real numbers, $N_{m+1}=N_m+\widetilde{N}_{m+1},m=0,1,\cdots$ are Hermitian matrices, the Hamiltonians $f_m\in\mathcal{Q}^{\Im}_{s_m,\mathcal{D}_m}$ and the new perturbations $q_{m+1}\in\mathcal{Q}^{\Re}_{s_{m+1}, \mathcal{D}_{m+1}}$.
Repeating iteratively the same product with $m+1$ instead of $m$, a change of variables $\Phi_{\epsilon}$  can be constructed which is $\epsilon_0^{\alpha_2}$-close to the identity mapping such that
\begin{equation*}
  (h_0+\sum_{m=0}^{\infty}\epsilon_mq_{0,m})\circ \Phi_{\epsilon}=h_{\infty},\ \ \ \ \omega\in\mathcal{D}_{\infty},
\end{equation*}
where $h_{\infty}=\langle z,N_{\infty}(\omega)\bar{z}\rangle+e_{\infty}(\omega)$ is in normal form with $e_{\infty}$ being of size $\mathcal{O}(\epsilon_0)$ and $N_{\infty}$ being Hermitian matrices for $(\theta,\omega)\in\mathbb{T}^n\times \mathcal{D}_{\infty}$ and $\epsilon_0$-close to $N_0$, and a compact set  $\mathcal{D}_{\infty}\subset\mathcal{D}$ is  $\gamma_0^{\alpha_1}$-close to $\mathcal{D}$.

\subsection{Homological equations}\label{sectionhomo}

In this subsection,  one search for a closed set $\mathcal{D}_1\subset\mathcal{D}$ and  a Hamiltonian $f_0\in\mathcal{Q}^{\Im}_{s_1,\mathcal{D}_1}$ such that (\ref{m}) holds for $m=0$.
Fix a real number $\tau>0 $ such that
\begin{equation}\label{ell}
\ell\geq 2n-1+\beta,\ \ \ \ \ n-1<\tau<n-1+\frac{\beta}{4},\ \ \ \ \ 0<\beta<1,
\end{equation}
and choose a small positive number $\rho$ satisfying
\begin{equation}\label{rho}
 0<\rho<\frac{\beta}{4(2n-1)+3\beta}.
\end{equation}
Choose $\epsilon_{\star}>0$ such that $\epsilon_{\star}\ll \gamma^{8}$, where
 \begin{equation}\label{con3}
  \ \ 0<\gamma_0=\gamma<\min\{3^{-2(n+3)}n^{-1}d^{-4},\nu_0, v_0\},\ \ \ \ \nu_0=\min_{i\neq j}\{|v_i-v_j|\}.
\end{equation}
For any $0<\epsilon_0<\epsilon_{\star}$,  let
\begin{equation}\label{con2}
  \epsilon_1=\epsilon_0^{1+\rho},\ \ s_0=\epsilon_{1}^{\frac{1}{\ell}}, s_1=\epsilon_{1}^{\frac{1+\rho}{\ell}},\ \ s^j_1=\frac{1}{4}(js_0+(4-j)s_1), j=1,2,3,\ \
  K_0=\frac{\log\epsilon_0^{-1}}{s_0-s^1_1}.
\end{equation}

Suppose $h_0$ and $q_0$  defined in (\ref{h0}) and (\ref{q01}) satisfy
\begin{equation}\label{con1}
\|Q^{z\bar{z}}_{0,0}\|_{s_0,\mathcal{D}},\ \|\partial_{\omega_{j}}Q^{z\bar{z}}_{0,0}\|_{s_0,\mathcal{D}}\leq C(0),\ \  j=1,\cdots,n,
\end{equation}
where $C(j)=C_12^{C_2j},j=0,1 $ with $C_1$ being a constant  depending on $n,\ell$ and $\gamma_0$ and $C_2$ being a constant  depending on $n$ and $\ell$. Let $q'_0=q_{0,0}$. The following proposition holds.

\begin{prop}\label{prop}
There exists a closed subset $\mathcal{D}_1\subset\mathcal{D}$ satisfying
\begin{equation}\label{measure}
  \meas(\mathcal{D}\setminus\mathcal{D}_1)\leq \gamma_0^{\frac{1}{2}},
\end{equation}
and there exist $f_0\in \mathcal{Q}^{\Im}_{s^1_1,\mathcal{D}_1}$, $r_1\in \mathcal{Q}^{\Re}_{s^1_1,\mathcal{D}_1}$ and a  matrix-valued $C^1$-mapping $\mathcal{D}_1\ni\omega\mapsto\widetilde{N}_1(\omega)\in\mathcal{M}_H$, the space of Hermitian matrices,  such that for all $\omega\in\mathcal{D}_1$,
\begin{equation}\label{1}
  \epsilon_0\{h_0,f_0\}+\epsilon_0q'_{0}-\epsilon_0\omega\cdot\partial_{\theta}f_0=\tilde{e}_1+\langle z,\widetilde{N}_{1}\bar z\rangle+r_{1}.
\end{equation}
Furthermore,
\begin{equation}\label{n1}
 \| \tilde{e}_1\|_{\mathcal{D}_1},\ \| \widetilde{N}_1\|_{\mathcal{D}_1},\ \|\partial_{\omega_l} \widetilde{N}_1\|_{\mathcal{D}_1}\leq C(0)\epsilon_0 , \ \ l=1,2,\cdots,n,
\end{equation}
\begin{equation}\label{f01}
  [f_0]^{(0)}_{s^1_1,\mathcal{D}_1}\leq C(1)\epsilon_0^{-\frac{(1+\rho)\tau}{\ell}},
\end{equation}
\begin{equation}\label{f002}
  [f_0]^{(1)}_{s^1_1,\mathcal{D}_1}
  \leq C(1)\epsilon_0^{-\frac{(1+\rho)(2\tau+1)}{\ell}},\ \ j=1,2,\cdots,n,
\end{equation}
and
\begin{equation}\label{r1}
   [r_1]_{s^1_1,\mathcal{D}_1}\leq C(1)\epsilon_0^{\frac{13}{8}}.
 \end{equation}
\end{prop}

To this end, let $f_0,r_1$ be in the form
\begin{equation*}
  f_0(\theta,z,\bar z)=\langle z,F_0^{zz}(\theta)z\rangle+\langle z,F_0^{z\bar z}(\theta)\bar z\rangle+\langle\bar z,{F}_0^{\bar{z}\bar{z}}(\theta)\bar z\rangle +\langle F_0^{z},z\rangle+\langle {F}_0^{\bar{z}},\bar{z}\rangle+F^{\theta}_0,
\end{equation*}
\begin{equation*}
  r_1(\theta,z,\bar z)=\langle z,R_1^{zz}(\theta)z\rangle+\langle z,R_1^{z\bar z}(\theta)\bar z\rangle+\langle\bar z,{R}_1^{\bar{z}\bar{z}}(\theta)\bar z\rangle +\langle R_1^{z},z\rangle+\langle {R}_1^{\bar{z}},\bar{z}\rangle+R^{\theta}_1.
\end{equation*}
From (\ref{possion}),
\begin{align*}
  \{h_0,f_0\}=&
  -\text{i}\langle N_0\bar{z},\left((F_0^{z\bar z})^{\top}z+({F}_0^{\bar z\bar z}+({F}_0^{\bar{z}\bar{z}})^{\top})\bar{z} +{F}_0^{\bar{z}}\right)\rangle\\
  &\ +\text{i}\langle N_0^{\top}z,\left((F_0^{z\bar{z}})\bar{z}+(F_0^{ zz}+(F_0^{zz})^{\top})z +F_0^{z}\right)\rangle\\
  =& \text{i}\langle z,[N_0,F_0^{z\bar{z}}]\bar{z}\rangle + \text{i}\langle z,(N_0F_0^{zz}+F_0^{zz}N_0^{\top})z\rangle
  \\
    &\ -\text{i}\langle \bar{z},\left(N_0^{\top}{F}_0^{\bar{z}\bar{z}} +{F}_0^{\bar{z}\bar{z}}N_0\right)\bar{z}\rangle+ \text{i}\langle N_0F_0^{z},z\rangle-\text{i}\langle N_0^{\top}{F}_0^{\bar{z}}, \bar{z}\rangle,
\end{align*}
where $[A,B]:=AB-BA$ for $d\times d$ matrices $A,B$.

To solve (\ref{1}),  consider the equation
\begin{equation}\label{homolg}
\{h_0,f_0\}-\omega\cdot\partial_{\theta}f_0+\Gamma_{K_0}q'_0=\langle z,[{Q}'^{z\bar{z}}_{0}]\bar z\rangle+[\hat{Q}'^{\theta}_{0}],
\end{equation}
where $q'_0=q_{0,0}$, $\Gamma_k$ denotes  the truncation operator,  $K_0$ is defined in (\ref{con2}),
$[{Q}^{z\bar{z}}_{0,0}]=\hat{Q}^{z\bar{z}}_{0,0}(0)$ and $\hat{Q}^{z\bar{z}}_{0,0}(k)$ is the $k$-th Fourier coefficient of ${Q}^{z\bar{z}}_{0,0}(\theta)$.
If (\ref{homolg}) is resolvable, then  the equation  (\ref{1}) holds by setting
\begin{equation}\label{N1}
\tilde{e}_1= \epsilon_0\hat{Q}^{\theta}_{0,0}(0),\ \  \widetilde{N}_{1}=\epsilon_0\hat{Q}^{z\bar{z}}_{0,0}(0),
\end{equation}
\begin{equation}\label{r1def}
r_1=\epsilon_0(1-\Gamma_{K_0})q_{0,0}=\epsilon_0\sum_{|k|>K_0}\hat{q}_{0,0}(k)e^{\text{i}\langle k,\theta\rangle}.
\end{equation}
It is easy to see that $e_1(\omega)=\tilde{e}_1(\omega)$ is real number, $N_1(\omega)=N_0+\widetilde{N}_1(\omega)$ is Hermitian matrix for any $\omega\in\mathcal{D}$ and $r_1\in \mathcal{Q}^{\Re}_{s_0,\mathcal{D}}$.

Comparing the  coefficients of both sides of (\ref{homolg}), the homological equation (\ref{1}) is equivalent to the following six equations:
\begin{equation}\label{homo1}
  -\omega\cdot\partial_{\theta}F_0^{z\bar{z}}+\text{i}[N_0,F_0^{z\bar{z}}]=[{Q}^{z\bar{z}}_{0,0}]-\Gamma_{K_0}Q^{z\bar{z}}_{0,0},
\end{equation}
\begin{equation}\label{homo2}
  -\omega\cdot\partial_{\theta}F_0^{zz}+\text{i}(N_0F_0^{zz}+F_0^{zz}N_0^{\top})=-\Gamma_{K_0}Q^{zz}_{0,0},
\end{equation}
\begin{equation}\label{homo2'}
  -\omega\cdot\partial_{\theta}F_0^{\bar z\bar z}-\text{i}(F_0^{\bar{z}\bar{z}}N_0+N_0^{\top}F_0^{\bar{z}\bar{z}})=-\Gamma_{K_0}Q^{\bar{z}\bar{z}}_{0,0},
\end{equation}
\begin{equation}\label{homo3}
  -\omega\cdot\partial_{\theta}F_0^{z}+\text{i}N_0F_0^{z}=-\Gamma_{K_0}Q^{z}_{0,0},
\end{equation}
\begin{equation}\label{homo3'}
  -\omega\cdot\partial_{\theta}F_0^{\bar{z}}-\text{i}N_0^{\top}F_0^{\bar{z}}=-\Gamma_{K_0}Q^{\bar{z}}_{0,0},
\end{equation}
\begin{equation}\label{homo4}
  -\omega\cdot\partial_{\theta}F_0^{\theta}=\hat{Q}^{\theta}_{0,0}(0)-\Gamma_{K_0}Q^{\theta}_{0,0}.
\end{equation}

In order to solve (\ref{homo1}), the small divisors shall be dealt with.
\begin{lem}\label{reverse}
Let $g:[0,1]\mapsto\mathbb{R}$ be  a $C^1$ map satisfying $|g'(x)|\geq \delta$ for all $x\in[0,1]$ and let $\kappa>0$. Then
\begin{equation*}
  \meas\{x\in[0,1]:|g(x)|\leq\kappa\}\leq \frac{\kappa}{\delta}.
\end{equation*}
\end{lem}

\begin{lem}\label{sumlambda}
Let $\mathcal{D}_{1,1}=\mathcal{D}_{1,11}\bigcap\mathcal{D}_{1,12}$, where
\begin{equation}\label{d1}
  \mathcal{D}_{1,11}=\{\omega\in\mathcal{D}: |\langle k,\omega\rangle-v_i+v_j| \geq\frac{\gamma_0}{1+|k|^{\tau}},\ i,j=1,\cdots,d,\   |k|\leq 2K_0\},
\end{equation}
\begin{equation}\label{d2}
  \mathcal{D}_{1,12}=\{\omega\in\mathcal{D}: |\langle k,\omega\rangle\pm v_i\pm v_j\pm v_{i'}\pm v_{j'}| \geq\frac{\gamma_0}{1+|k|^{\tau}},\ i,j,i',j'=1,\cdots,d,\  |k|\leq 2K_0\}.
\end{equation}
Then
\begin{equation}\label{measure1}
  \meas(\mathcal{D}\setminus\mathcal{D}_{1,1})\leq \frac{1}{3}\gamma_0^{\frac{1}{2}},
\end{equation}
and for $\omega\in\mathcal{D}_{1,1}$, it appears
\begin{equation*}
 \sum_{0<|k|\leq m}\frac{1}{D_k(\omega)^2} \leq 2^{2n+3}\frac{m^{2\tau}}{\gamma_0^2},\ \ m=1,2,\cdots,K_0,
\end{equation*}
where $D_k(\omega)=\min\{|\langle k,\omega\rangle-v_i+v_j|:i,j=1,\cdots,d\}$.
\end{lem}

\begin{proof}
Denote
\begin{equation}\label{d'}
  \mathcal{R}_{0,kij}=\{\omega\in\mathcal{D}:|\lambda_{kij}(\omega)|<\frac{\gamma_0}{1+|k|^{\tau}}\},\ \ \ \ \ |k|\leq 2K_0, \ \ \ \  i,j=1,\cdots,d,
\end{equation}
where
\begin{equation*}
 \lambda_{k}(\omega)= \lambda_{kij}(\omega)=\langle k,\omega\rangle-v_i+v_j.
\end{equation*}
From (\ref{con3}), $\meas(\mathcal{R}_{0,0ij})=0$ for $i\neq j$.
Fixing a direction $z_k=(z_k^{(1)},\cdots,z_k^{(n)})\in\mathbb{R}^n$ with $z_k^{(j)}=\text{sign}(k_j), j=1,\cdots,n$, it is deduced that
\begin{equation}\label{lambda0}
  \partial_{z_k}\lambda_k(\omega)=\sum_{j=1}^n\partial_{\omega_j}\lambda_{kij}(\omega)\cdot \frac{z_k^{(j)}}{|z_k|_2}=\frac{|k|}{|z_k|_2}\geq\frac{|k|}{\sqrt{n}},\ \ \ \ \ k\neq0.
\end{equation}
Combining  (\ref{d'}) and Lemma \ref{reverse},  one has
\begin{equation}\label{meas}
  \meas(\mathcal{R}_{0,kij})\leq \frac{\sqrt{n}\gamma_0}{|k|+|k|^{\tau+1}}\leq \frac{2\sqrt{n}\gamma_0}{|k|^{\tau+1}},\ \ \ \ \ i,j=1,\cdots,d,\ \ \ \ 0<|k|\leq 2K_0,
\end{equation}
and therefore from (\ref{con3}),
\begin{align*}
 \meas({\mathcal{D}\setminus\mathcal{D}_{1,11}})&=\meas(\bigcup_{0<|k|\leq 2K_0,i,j}\mathcal{R}_{0,kij})
 \leq d^2\sum_{0<|k|\leq 2K_0}\meas{\mathcal{R}_{0,kij}}\\
 &\leq 2n^{\frac{1}{2}}d^2\gamma_0\sum_{l=1}^{2K_0}\sum_{|k|=l}|k|^{-(\tau+1)}
  \leq 3^{n}n^{\frac{1}{2}}d^2\gamma_0\sum_{l=1}^{2K_0}l^{-(\tau-n+2)}\leq \frac{1}{6}\gamma_0^{\frac{1}{2}}.
\end{align*}
Similarly, $\meas({\mathcal{D}\setminus\mathcal{D}_{1,12}})\leq \frac{1}{6}\gamma_0^{\frac{1}{2}}$, and (\ref{measure1}) holds.

Fixing $\omega\in\mathcal{D}_{1,1}$, one can rank all the non-negative numbers $D_{k}(\omega),\ 0<|k|\leq m$ in order of size:
$D_{k^{(1)}}\leq D_{k^{(2)}}\leq \cdots\leq D_{k^{(M)}}$, where $M$ is the number of all possible $|k|\leq m$. Denote $D_{k^{(l)}}=|\langle k^{(l)},\omega\rangle-v_{i^{(l)}}+v_{j^{(l)}}|,\ 1\leq l\leq M$. From (\ref{d1}) and $|k^{(2)}\pm k^{(1)}|\leq 2m\leq 2K_0$, it appears
\begin{equation*}
  D_{k^{(2)}}-D_{k^{(1)}}= |\langle k^{(2)}\pm k^{(1)},\omega\rangle-v_{i^{(2)}}+v_{j^{(2)}}\pm (-v_{i^{(1)}}+v_{j^{(1)}})|
  \geq \frac{\gamma_0}{1+|k^{(2)}\pm k^{(1)}|^{\tau}}\geq \frac{\gamma_0}{1+(2m)^{\tau}},
\end{equation*}
and therefore
\begin{equation*}
   D_{k^{(2)}}\geq \frac{\gamma_0}{1+(2m)^{\tau}}+ D_{k^{(1)}}\geq \frac{\gamma_0}{1+(2m)^{\tau}}+\frac{\gamma_0}{1+m^{\tau}}\geq \frac{2\gamma_0}{1+(2m)^{\tau}}.
\end{equation*}
Repeating the process above, one has
\begin{equation*}
   D_{k^{(l)}}\geq \frac{l\gamma_0}{1+(2m)^{\tau}},\ \ \ \ \  l=1,2,\cdots,M.
\end{equation*}
Therefore, from (\ref{ell}),
\begin{equation}
\sum_{0<|k|\leq m}\frac{1}{ D_{k}(\omega)^2} = \sum_{l=1}^{M}\frac{1}{ D_{k^{(l)}}^2}   \leq 2^{2n+3}\frac{m^{2\tau}}{\gamma_0^2}.
\end{equation}
\end{proof}

The following four lemmata are useful.
\begin{lem}[Lemma 2.1 in \cite{russmann1975}]\label{norm}
Let $f$ be a complex valued function, which is defined and analytic in the strip $|\Im x|=\max_{1\leq j\leq n}|\Im x_j|<r$ and has a period $2\pi$ in each variable $x_1,\cdots,x_n$. If
\begin{equation*}
  \|f\|_{r}=\sup_{|\Im x|<r}|f(x)|\leq M,
\end{equation*}
with some positive constant $M$, then  one has the inequality
\begin{equation*}
  \sum_{k\in\mathbb{Z}^n}|\hat{f}(k)|^2e^{2|k|r}\leq 2^nM^2,
\end{equation*}
where $|k|=|k_1|+\cdots+|k_n|$ for $k\in \mathbb{Z}^n$.
\end{lem}

\begin{lem}[Lemma 2.2 in \cite{russmann1975}]\label{phist}
Let $\varphi$ be a real function, which is defined and continuous in the interval $[0,\infty[$ such that
\begin{equation*}
  0\leq \varphi(s)\leq \varphi(t),\ \ \ \ \ 0\leq s\leq t,
\end{equation*}
\begin{equation*}
  \lim_{s\rightarrow \infty} s^{-1}\log \varphi(s)=0.
\end{equation*}
Furthermore let $a_0,a_1,\cdots$ be a sequence of nonnegative real numbers such that
\begin{equation*}
  \sum_{j=0}^{m}a_{j}\leq\varphi(m),\ \ \ \ \ m=0,1,\cdots.
\end{equation*}
Then for every $\delta>0$, one gets
\begin{equation*}
  \sum_{j=0}^{\infty}a_{j}e^{-\delta j}\leq \int_{0}^{\infty}e^{-s}\varphi\left(\frac{s}{\delta}\right)\text{d}s<\infty.
\end{equation*}
\end{lem}

\begin{lem}[Cauchy's estimate]\label{cauchyes}
Let $\mathcal{D}$ be an open domain in $\mathbb{C}^n$, let $\mathcal{D}_r=\{z:|z-\mathcal{D}|<r\}$ be the neighbourhood of radius $r$ around $\mathcal{D}$, and let $g$ be an analytic function on $\mathcal{D}_r$ with bounded sup-norm $\|g\|_r$. Then
\begin{equation*}
  \|g_{z_j}\|_{r-\rho}\leq \frac{1}{\rho}\|g\|_r
\end{equation*}
for all $0<\rho<r$ and $1\leq j\leq n$.
\end{lem}

\begin{lem}\label{tenser}
Let $A,\ B$ and $C$ be $r\times r$, $s\times s$ and $r\times s$ matrices respectively, and let $X$ be an  $r\times s$ unknown matrix. Then the matrix equation
\begin{equation*}
  AX+XB=C
\end{equation*}
is solvable if and only if the vector equation
\begin{equation*}
 ( I_s\otimes A+B^{\top}\otimes I_r)X^{\dag}=C^{\dag}
\end{equation*}
is solvable, where $X^{\dag}=(X_1^{\top},\cdots,X_s^{\top})^{\top}$ and $C^{\dag}=(C_1^{\top},\cdots,C_s^{\top})^{\top}$ if one writes $X=(X_1,\cdots,X_s)$ and $C=(C_1,\cdots,C_s)$. The spectrum of $I_s\otimes A+B^{\top}\otimes I_r$ is $\{\lambda_i+\mu_j|i=1,\cdots,r,j=1,\cdots,s\}$  if $\sigma(A)=\{\lambda_1,\cdots,\lambda_r\}$ and $\sigma(B)=\{\mu_1,\cdots,\mu_s\}$. Moreover,
\begin{equation*}
  \|X\|\leq \|(I_s\otimes A+B^{\top}\otimes I_r)^{-1}\|\|C\|
\end{equation*}
if the inverse exists, where $\|\cdot\|$ is the operator norm of the matrix.
\end{lem}

Now one can solve the equation (\ref{homo1}).
\begin{lem}\label{pro1}
For the set $\mathcal{D}_{1,1}$ defined in Lemma \ref{sumlambda}, the equation (\ref{homo1}) has a unique solution $\mathbb{T}^n_{s^1_1}\times\mathcal{D}_{1,1}\ni(\theta,\omega)\mapsto F_0^{z\bar{z}}(\theta;\omega)$ which is analytic in $\theta\in\mathbb{T}^n_{s^1_1}$ and $C^1$ in $\omega\in\mathcal{D}_{1,1}$  with the estimates
\begin{equation}\label{F01}
  \|{F}^{z\bar{z}}_0\|_{s^1_1,\mathcal{D}_{1,1}}\leq C(1)\epsilon_0^{-\frac{(1+\rho)\tau}{\ell}},
\end{equation}
\begin{equation}\label{F01'}
  \|\partial_{\omega_l}F^{z\bar{z}}_0\|_{s^1_1,\mathcal{D}_{1,1}}\leq C(1)\epsilon_0^{-\frac{(1+\rho)(2\tau+1)}{\ell}},\ \ l=1,2,\cdots,n.
\end{equation}
Moreover, for any $(\theta,\omega)\in\mathbb{T}^n\times\mathcal{D}_{1,1}$, $F_0^{z\bar{z}}(\theta;\omega)$ is Hermitian matrix.
\end{lem}

\begin{proof}
Representing the matrix elements of $F_0^{z\bar{z}}(\theta),Q_{0,0}^{z\bar{z}}(\theta)$ in Fourier series with respect to $\theta\in\mathbb{T}^n_{s^1_1}$,  the equation (\ref{homo1}) is equivalent to
\begin{equation*}
  F^{z\bar{z}}_{0}(k)=0,\ \ \ \ \ k=0 \text{ or } |k|>K_0,
\end{equation*}
\begin{equation}\label{homo11}
  -\text{i}\langle k,\omega\rangle \hat{F}_0^{z\bar{z}}(k)+\text{i}[N_0,\hat{F}_0^{z\bar{z}}(k)]= -\hat{Q}_{0,0}^{z\bar{z}}(k),\ \ \ \ \ 0<|k|\leq K_0.
\end{equation}

For $0<|k|\leq K_0$,  rewrite (\ref{homo11}) as
\begin{equation}\label{homo12}
  (\langle k,\omega\rangle I_d-N_0)\hat{F}_0^{z\bar{z}}(k)+\hat{F}_0^{z\bar{z}}(k)N_0 =-\text{i}\hat{Q}_{0,0}^{z\bar{z}}(k).
\end{equation}
From Lemma \ref{tenser}, the equation (\ref{homo12}) is solvable if and only if the vector equation
\begin{equation}\label{tenser1}
  (I_{d}\otimes(\langle k,\omega\rangle I_d-N_0)+N_0^{\top}\otimes I_d)(\hat{F}_0^{z\bar{z}}(k))^{\dag}=-\text{i}(\hat{Q}_{0,0}^{z\bar{z}}(k))^{\dag}
\end{equation}
is solvable, where
\begin{equation*}
(\hat{F}_0^{z\bar{z}}(k))^{\dag}:=((\hat{F}_{0,1}^{z\bar{z}}(k))^{\top},\cdots,(\hat{F}_{0,d}^{z\bar{z}}(k))^{\top})^{\top}, \ (\hat{Q}_{0,0}^{z\bar{z}}(k))^{\dag}:=((\hat{Q}_{0,0,1}^{z\bar{z}}(k))^{\top},\cdots,(\hat{Q}_{0,0,d}^{z\bar{z}}(k))^{\top})^{\top}
\end{equation*}
if one writes $\hat{F}_{0}^{z\bar{z}}(k)=(\hat{F}_{0,1}^{z\bar{z}}(k),\cdots,\hat{F}_{0,d}^{z\bar{z}}(k))$ and $\hat{Q}_{0,0}^{z\bar{z}}(k)=(\hat{Q}_{0,0,1}^{z\bar{z}}(k),\cdots,\hat{Q}_{0,0,d}^{z\bar{z}}(k))$.
For $\omega\in\mathcal{D}_1$,  denote $L_k(\omega)=I_{d}\otimes(\langle k,\omega\rangle I_d-N_0)+N_0^{\top}\otimes I_d$. Then the spectrum of $L_k(\omega)$ is composed of real eigenvalues
\begin{equation}
\sigma(L_k(\omega))=\{\langle k,\omega\rangle-v_i+v_j:v_i,v_j\in\sigma(N_0),\  i,j=1,\cdots,d\}.
\end{equation}
From the definition of $\mathcal{D}_{1,1}$ in Lemma \ref{sumlambda}, all the eigenvalues of $L_k(\omega)$ are non-zero, hence the equation (\ref{homo12}) is resolvable. By using Cauchy-Schwarz inequality and Lemma \ref{norm}, the solution $ {F}_0^{z\bar{z}}(\theta;\omega)$ of (\ref{homo11}) satisfies
\begin{align}\label{f0}
\nonumber \| {F}_0^{z\bar{z}}\|_{s^1_1,\mathcal{D}_{1,1}}
&\leq \sum_{0<|k|\leq K_0} \|L^{-1}_k\| \|\hat{Q}_{0,0}^{z\bar{z}}(k)\|_{\mathcal{D}} e^{|k|s^1_1} \\
\nonumber &\leq \left(\sum_{k\in\mathbb{Z}^n}\|\hat{Q}^{z\bar{z}}_{0,0}(k)\|_{\mathcal{D}}^2e^{2|k|s_0}\right)^{\frac{1}{2}}  \left(\sum_{0<|k|\leq K_0}\|L^{-1}_k(\omega)\|^2e^{-2|k|(s_0-s^1_1)}\right)^{\frac{1}{2}} \\
 &\leq 2^{\frac{n}{2}}\|Q^{z\bar{z}}_{0,0}\|_{s_0,\mathcal{D}}\left(\sum_{0<|k|\leq K_0} \frac{1}{D_k(\omega)^2}e^{-2|k|(s_0-s^1_1)}\right)^{\frac{1}{2}}.
\end{align}
Let
\begin{equation}\label{amu1}
  a_{m}=0,\ \ \ \ \  m=0 \text{ or }m>K_0,
\end{equation}
\begin{equation}\label{amu2}
a_{m}=\sum_{|k|=m}\frac{1}{D_k(\omega)^2},\ \ \ \ \ m=1,2,\cdots,K_0.
\end{equation}
Then
\begin{equation}\label{f000}
  \| {F}_0^{z\bar{z}}\|_{s^1_1,\mathcal{D}_{1,1}}\leq 2^{\frac{n}{2}}\|Q^{z\bar{z}}_{0,0}\|_{s_0,\mathcal{D}} \left(\sum_{m=0}^{K_0}a_{m}e^{-2m(s_0-s^1_1)}\right)^{\frac{1}{2}}.
\end{equation}
Let $\varphi(m)=2^{2n+3}\frac{m^{2\tau}}{\gamma_0^2}$.
From Lemma \ref{sumlambda}, the function $\varphi$  and the sequence  $\{a_{m}\}_{m=0}^{\infty}$ satisfy the conditions of Lemma \ref{phist}. As a result,
\begin{align*}
 \sum_{m=0}^{K_0}a_{m}e^{-2m(s_0-s^1_1)}&=\sum_{m=0}^{\infty}a_{m}e^{-2m(s_0-s^1_1)}\\
                &\leq\int_{0}^{\infty}e^{-t}\varphi\left(\frac{t}{2(s_0-s^1_1)}\right)\text{d}s\\
                &\leq \frac{2^{2n+3}}{\gamma_0^{2}}\frac{1}{[2(s_0-s^1_1)]^{2\tau}}\int_{0}^{\infty}t^{2\tau}e^{-t}\text{d}t\\
                &\leq \frac{2^{2n+3}\Gamma(2\tau+1)}{\gamma_0^{2}}\frac{1}{[2(s_0-s^1_1)]^{2\tau}}.
\end{align*}
Combining  (\ref{f0}) and the conditions (\ref{con1}) and (\ref{con2}),  one has
\begin{equation}\label{F0}
\| {F}_{0}^{z\bar{z}}\|_{s^1_1,\mathcal{D}_{1,1}}
 \leq \frac{2^{\frac{3n}{2}+\frac{3}{2}}\sqrt{(2n-1)!}}{\gamma_0[2(s_0-s^1_1)]^{\tau}}\|Q^{z\bar{z}}_{0,0}\|_{s_0,\mathcal{D}}
 \leq C(1)\epsilon_0^{-\frac{(1+\rho)\tau}{\ell}}.
\end{equation}

In the following,  estimation on the norm of $\partial_{\omega_j}\hat{F}_0^{z\bar{z}}(k)$ is given.
For $\omega\in\mathcal{D}_{1,1}$, differentiating both side of (\ref{homo12}) with respect to $\omega$ leads to
\begin{equation}\label{lkomega}
 (\langle k,\omega\rangle I_d-N_0)\partial_{\omega_l}\hat{F}_0^{z\bar{z}}(k)+\partial_{\omega_l}\hat{F}_0^{z\bar{z}}(k)N_0 =-\text{i}\partial_{\omega_l}\hat{Q}_{0,0}^{z\bar{z}}(k)-k_l\hat{F}_0^{z\bar{z}}(k),\ \ \ \ \ l=1,\cdots,n,\ \ \ \ 0<|k|\leq K_0.
\end{equation}
From (\ref{homo12}), Lemma \ref{tenser}, Lemma \ref{norm} and by using Cauchy-Schwartz inequality, it appears
\begin{align*}
 \|\partial_{\omega_l}{F}^{z\bar{z}}_{0}\|_{s^1_1,\mathcal{D}_{1,1}}
 &\ \leq\sum_{0<|k|\leq K_0} \frac{|\partial_{\omega_l}\hat{Q}^{z\bar z}_{0,0}(k)|e^{|k|s^1_1}}{D_{k}(\omega)}
  +\sum_{0<|k|\leq K_0}\frac{|k_l||\hat{F}^{z\bar z}_{0}(k)|e^{|k|s^1_1}}{D_{k}(\omega)}\\
&\  \leq\sum_{0<|k|\leq K_0} \frac{|\partial_{\omega_l}\hat{Q}^{z\bar z}_{0,0}(k)|e^{|k|s^1_1}}{D_{k}(\omega)}
  +\sum_{0<|k|\leq K_0}\frac{|k_l||\hat{Q}^{z\bar z}_{0,0}(k)|e^{|k|s^1_1}}{D_{k}(\omega)^2}\\
&\  \leq 2^{\frac{n}{2}}\|\partial_{\omega_l}Q^{z\bar{z}}_{0,0}\|_{s_0,\mathcal{D}}\left(\sum_{0<|k|\leq K_0} \frac{1}{D_k(\omega)^2}e^{-2|k|(s_0-s^1_1)}\right)^{\frac{1}{2}}\\
 &\ \ \ \ \  + 2^{\frac{n}{2}}\|Q^{z\bar{z}}_{0,0}\|_{s_0,\mathcal{D}}\left(\sum_{0<|k|\leq K_0} \frac{|k_l|^2}{D_k(\omega)^4}e^{-2|k|(s_0-s^1_1)}\right)^{\frac{1}{2}},\ \ \ \ l=1,\cdots,n.
\end{align*}
Similar to Lemma \ref{sumlambda}, one can prove that
\begin{equation*}
 \sum_{0<|k|\leq m}\frac{|k|^2}{D_k(\omega)^4} \leq 2^{4n+6}\frac{m^{4\tau+2}}{\gamma_0^4},\ \ m=1,2,\cdots,K_0.
\end{equation*}
Therefore, from Lemma \ref{phist}, one has
\begin{align*}
 \|\partial_{\omega_l}{F}^{z\bar{z}}_{0}\|_{s^1_1,\mathcal{D}_{1,1}}
  &\leq \frac{C(1)}{[2(s_0-s^1_1)]^{\tau}}\|\partial_{\omega_l}Q^{z\bar{z}}_{0,0}\|_{s_0,\mathcal{D}}
  + \frac{C(1)}{[2(s_0-s^1_1)]^{2\tau+1}}\|Q^{z\bar{z}}_{0,0}\|_{s_0,\mathcal{D}}\\
  &\leq C(1)\epsilon_0^{-\frac{(1+\rho)(2\tau+1)}{\ell}},  \ \ \ \ \ l=1,\cdots,n.
\end{align*}

Moreover, since for $(\theta,\omega)\in\mathbb{T}^n\times \mathcal{D}_1$, $Q^{z\bar{z}}_{0,0}(\theta)$ is Hermitian matrix, one has $\hat{Q}^{z\bar{z}}_{0,0}(k)=(\hat{Q}^{z\bar{z}}_{0,0}(-k))^{\ast}$  by comparing the Fourier coefficients of $Q^{z\bar{z}}_{0,0}(\theta)$, where $A^{\ast}$ denotes the conjugate transpose of a matrix $A$. The equation (\ref{homo12}) holds true when replacing $k$ by  $-k$, that is,
\begin{equation}\label{-k}
  (\langle -k,\omega\rangle I_d-N_0) \hat{F}_0^{z\bar{z}}(-k)+N_0\hat{F}_0^{z\bar{z}}(-k)= -\text{i}\hat{Q}_{0,0}^{z\bar{z}}(-k),\ \ \ \ \ 0<|k|\leq K_0.
\end{equation}
Taking both side of (\ref{-k}) the conjugate transposes, one has
 \begin{equation}\label{ast}
  (\langle -k,\omega\rangle I_d-N_0) (\hat{F}_0^{z\bar{z}}(-k))^{\ast}+N_0(\hat{F}_0^{z\bar{z}}(-k))^{\ast}= -\text{i}\hat{Q}_{0,0}^{z\bar{z}}(k),\ \ \ \ \ 0<|k|\leq K_0.
\end{equation}
Since $F^{z\bar z}_0(\theta)$ is the unique solution of (\ref{homo12}), one has $\hat{F}_0^{z\bar{z}}(k)=(\hat{F}_0^{z\bar{z}}(-k))^{\ast}, 0<|k|\leq K_0$, which implies $F^{z\bar{z}}_{0}(\theta)$ is Hermitian matrix for any $(\theta,\omega)\in\mathbb{T}^n\times \mathcal{D}_1$.
The lemma is proved.
\end{proof}

Notice that  for the equation (\ref{homo4}) one meets the small divisors $\langle k,\omega\rangle$. By letting $\hat{F}^{\theta}_0(0)=0$ and $\hat{F}^{\theta}_{0}(k)=\frac{-\text{i}\hat{Q}^{\theta}_{0,0}(k)}{\langle k,\omega\rangle},0<|k|\leq K_0$, one has the following lemma
\begin{lem}\label{pro4}
For the set $\mathcal{D}_{1,1}$ defined in  Lemma \ref{sumlambda}, the equation (\ref{homo4})  has an unique solution $\mathbb{T}^n_{s^1_1}\times\mathcal{D}_{1,1}\ni(\theta,\omega)\mapsto F_0^{\theta}(\theta;\omega)$ which is analytic in $\theta\in\mathbb{T}^n_{s^1_1}$ and $C^1$ in $\omega\in\mathcal{D}_{1,1}$ with
\begin{equation}\label{F04}
  \|F^{\theta}_0\|_{s^1_1,\mathcal{D}_{1,2}}\leq C(1)\epsilon_0^{-\frac{(1+\rho)\tau}{\ell}},
\end{equation}
\begin{equation}\label{F04'}
  \|\partial_{\omega_l}F^{\theta}_0\|_{s^1_1,\mathcal{D}_{1,2}}\leq C(1)\epsilon_0^{-\frac{(1+\rho)(2\tau+1)}{\ell}},\ \ \ \ \ l=1,2,\cdots,n.
\end{equation}
Moreover, for any $(\theta,\omega)\in\mathbb{T}^n\times\mathcal{D}_{1,2}$, $F_0^{\theta}(\theta;\omega)$ is a pure imaginary number.
\end{lem}

Now it is  turned to consider the equations (\ref{homo3}) and (\ref{homo3'}). By representing  $F_0^{z}(\theta)$, $F_0^{\bar z}(\theta)$, $Q_{0,0}^{z}(\theta)$, $Q_{0,0}^{\bar{z}}(\theta)$ in Fourier series with respect to $\theta\in\mathbb{T}^n_{s^1_1}$, one has
\begin{equation}\label{homo121}
  -\text{i}\langle k,\omega\rangle \hat{F}_0^{z}(k)+\text{i}N_0\hat{F}_0^{z}(k)= -\hat{Q}_{0,0}^{z}(k),\ \ \ \ \ 0\leq|k|\leq K_0,
\end{equation}
\begin{equation}\label{homo122}
  -\text{i}\langle k,\omega\rangle \hat{F}_0^{\bar{z}}(k)-\text{i}N_0^{\top}\hat{F}_0^{\bar{z}}(k)= -\hat{Q}_{0,0}^{\bar{z}}(k),\ \ \ \ \ 0\leq|k|\leq K_0.
\end{equation}
In these cases  one meets the matrices $I_{d}\otimes(\langle k,\omega\rangle I_d- N_0)$ and  $I_{d}\otimes(\langle k,\omega\rangle I_d+ N_0^{\top})$ and the eigenvalues  $|\langle k,\omega\rangle\mp v_i|,\ i=1,\cdots,d$. In particular for $k=0$, these eigenvalues are $\geq v_0>0$.
Moreover, since for $\theta\in\mathbb{T}^n$, $\omega\in\mathcal{D}$, $Q_{0,0}^{\bar z}(\theta)$ is the conjugate vector of $Q_{0,0}^{z}(\theta)$, one has that $\hat{Q}^{\bar z}_{0,0}(-k)$ is the conjugate vector  of $\hat{Q}_{0,0}^{z}(k)$  for any $0\leq|k|\leq K_0$. Taking both side of (\ref{homo122}) the conjugate  and replacing $k$ by  $-k$, it appears
\begin{equation}\label{homo122'}
  -\text{i}\langle k,\omega\rangle \overline{\hat{F}_0^{\bar{z}}}(-k)+\text{i}N_0\overline{\hat{F}_0^{\bar{z}}}(-k)= -\hat{Q}_{0,0}^{{z}}(k),\ \ \ \ \ 0\leq|k|\leq K_0.
\end{equation}
 Similar to Lemma \ref{pro1},  by choosing a suitable set $\mathcal{D}'\subset\mathcal{D}$, $F^{z}_0(\theta)$ is the unique solution of (\ref{homo121}), one has $\hat{F}_0^{z}(k)=\overline{\hat{F}_0^{\bar{z}}}(-k), 0<|k|\leq K_0$, which implies $F_0^{\bar{z}}(\theta;\omega)$ is the conjugate vector of $F_0^{{z}}(\theta;\omega)$ for any $(\theta,\omega)\in\mathbb{T}^n\times\mathcal{D}'$.
One has the following result.

\begin{lem}\label{pro2}
There exists a closed subset $\mathcal{D}_{1,2}\subset\mathcal{D}$ with
\begin{equation}\label{measure2}
  \meas(\mathcal{D}\setminus\mathcal{D}_{1,2})\leq \frac{1}{3}\gamma_0^{\frac{1}{2}}
\end{equation}
such that the equations (\ref{homo3}) and (\ref{homo3'})  have  unique solutions $\mathbb{T}^n_{s^1_1}\times\mathcal{D}_{1,2}\ni(\theta,\omega)\mapsto F_0^{z}(\theta;\omega)$, $\mathbb{T}^n_{s^1_1}\times\mathcal{D}_{1,2}\ni(\theta,\omega)\mapsto F_0^{\bar{z}}(\theta;\omega)$, which are analytic in $\theta\in\mathbb{T}^n_{s^1_1}$ and $C^1$ in $\omega\in\mathcal{D}_{1,2}$ with the estimates
\begin{equation}\label{F02}
  \|F^{z}_0\|_{s^1_1,\mathcal{D}_{1,2}},\ \|F^{\bar{z}}_0\|_{s^1_1,\mathcal{D}_{1,2}}\leq C(1)\epsilon_0^{-\frac{(1+\rho)\tau}{\ell}},
\end{equation}
\begin{equation}\label{F02'}
  \|\partial_{\omega_l}F^{z}_0\|_{s^1_1,\mathcal{D}_{1,2}},\ \|\partial_{\omega_l}F^{\bar{z}}_0\|_{s^1_1,\mathcal{D}_{1,2}}\leq C(1)\epsilon_0^{-\frac{(1+\rho)(2\tau+1)}{\ell}},\ \ \ \ \ l=1,2,\cdots,n.
\end{equation}
Moreover, for any $(\theta,\omega)\in\mathbb{T}^n\times\mathcal{D}_{1,2}$, $F_0^{\bar{z}}(\theta;\omega)$ is the conjugate vector of $F_0^{{z}}(\theta;\omega)$.
\end{lem}

For the equations (\ref{homo2}), one can consider the matrices $I_d\otimes(\langle k,\omega\rangle I_d\mp N_0)\mp N_0^{\top}\otimes I_d$ with the eigenvalues $|\langle k,\omega\rangle\mp(v_i+v_j)|$.
These reduce to $|v_i+v_j|\geq 2v_0$ when $k=0$. One has
\begin{lem}\label{pro3}
There exists a closed subset $\mathcal{D}_{1,3}\subset\mathcal{D}$ with
\begin{equation}\label{measure3}
  \meas(\mathcal{D}\setminus\mathcal{D}_{1,3})\leq \frac{1}{3}\gamma_0^{\frac{1}{2}}
\end{equation}
such that the equations (\ref{homo2}) and (\ref{homo2'}) have  unique solutions  $\mathbb{T}^n_{s^1_1}\times\mathcal{D}_{1,3}\ni(\theta,\omega)\mapsto F_0^{zz}(\theta;\omega)$, $\mathbb{T}^n_{s^1_1}\times\mathcal{D}_{1,3}\ni(\theta,\omega)\mapsto F_0^{\bar{z}\bar{z}}(\theta;\omega)$, which are analytic in $\theta\in\mathbb{T}^n_{s^1_1}$ and $C^1$ in $\omega\in\mathcal{D}_{1,2}$ with the estimates
\begin{equation}\label{F03}
  \|F^{zz}_0\|_{s^1_1,\mathcal{D}_{1,3}},\|F^{\bar{z}\bar{z}}_0\|_{s^1_1,\mathcal{D}_{1,3}}\leq C(1)\epsilon_0^{-\frac{(1+\rho)\tau}{\ell}},
\end{equation}
\begin{equation}\label{F03'}
  \|\partial_{\omega_l}F^{zz}_0\|_{s^1_1,\mathcal{D}_{1,3}},\|\partial_{\omega_l}F^{\bar{z}\bar{z}}_0\|_{s^1_1,\mathcal{D}_{1,3}}
  \leq C(1)\epsilon_0^{-\frac{(1+\rho)(2\tau+1)}{\ell}} ,\ \ \ \ \ l=1,2,\cdots,n.
\end{equation}
Moreover, for any $(\theta,\omega)\in\mathbb{T}^n\times\mathcal{D}_{1,3}$, $F_0^{\bar{z}\bar{z}}(\theta;\omega)$ is the conjugate matrix of $F_0^{zz}(\theta;\omega)$.
\end{lem}

Now  Proposition \ref{prop} can be proved.

\begin{proof}[Proof of Proposition \ref{prop}]
Let $\mathcal{D}_1=\bigcap_{j=1}^3\mathcal{D}_{1,j}$. From Lemmata \ref{pro1}-\ref{pro3}, the definitions of $\tilde{e}_1$ and $\widetilde{N}_1$ in (\ref{N1}), the definitions of $\mathcal{Q}^{\Re}_{s^1_1,\mathcal{D}_1}, \mathcal{Q}^{\Im}_{s^1_1,\mathcal{D}_1}$ in (\ref{qsdr}) and (\ref{qsdi}), one has  $e_1(\omega)=e_0+\tilde{e}_1(\omega)=\tilde{e}_1(\omega)$ is real number, $N_1(\omega)=N_0+\widetilde{N}_1(\omega)$ is Hermitian matrix for any $\omega\in\mathcal{D}_1$,  $f_0\in\mathcal{Q}^{\Im}_{s^1_1,\mathcal{D}_1}$, and
the inequalities  (\ref{measure}), (\ref{n1})-(\ref{f002}) hold.
From  the definition of $r_1$ in (\ref{r1def}), the condition (\ref{con2}) and Lemma \ref{norm}, one gets $r_1\in\mathcal{Q}^{\Re}_{s^1_1,\mathcal{D}_1}$ and the estimates
\begin{align*}
\|r_1^{z\bar{z}}\|_{s^1_1,\mathcal{D}_{1}}& \leq\epsilon_0\sum_{|k|>K_0}\|\hat{Q}^{z\bar{z}}_{0,0}(k)\|_{\mathcal{D}}e^{|k|s^1_1}\\
&\leq  2^{\frac{n}{2}}\epsilon_0\|Q^{z\bar{z}}_{0,0}\|_{s_0,\mathcal{D}}\left(\sum_{|k|>K_0}e^{-2|k|(s_0-s^1_1)}\right)^{\frac{1}{2}}\\
&\leq 2^{n}C(0)\epsilon_0K_0^{\frac{n}{2}}e^{-K_0(s_0-s^1_1)}\\
&\leq C(1)(\log\epsilon_0^{-1})^{\frac{n}{2}}\epsilon_0^{2-\frac{(1+\rho)n}{2\ell}}.
\end{align*}
 One can choose $\epsilon_{\star}$ small enough, such that $\log \epsilon_0^{-1}\leq \epsilon_0^{-\frac{3}{4n}+\frac{1+\rho}{\ell}}$.  It follows that
\begin{equation*}
\|r_1^{z\bar{z}}\|_{s^1_1,\mathcal{D}_{1}} \leq  C(1)\epsilon_0^{1+\frac{5}{8}}\leq   C(1)\epsilon_0^{1+\rho},
\end{equation*}
Similarly,
\begin{equation*}
 \|\partial_{\omega_l}r_1^{z\bar{z}}\|_{s^1_1,\mathcal{D}_{1}}
\leq\epsilon_0\sum_{|k|>K_0}\|\partial_{\omega_l}\hat{Q}^{z\bar{z}}_{0,0}(k;\omega)\|e^{-|k|s^1_1}
\leq C(1)\epsilon_0^{1+\frac{5}{8}}\leq C(1)\epsilon_0^{1+\rho},\ \ l=1,\cdots,n.
\end{equation*}
The estimates for other terms of $r_1$ are similar.
Proposition \ref{prop} is proved.
\end{proof}

Noting that in the $m$-th KAM step ($m\geq 1$), one meets the $C^1$-function $\mathcal{D}_m\ni\omega\mapsto N_m(\omega)\in\mathcal{M}_H$, where
\begin{equation*}
N_m(\omega)=N_{m-1}(\omega)+\widetilde{N}_m(\omega)=N_0+\sum_{j=1}^m\widetilde{N}_j(\omega)
\end{equation*}
is $\epsilon_0$-close to $N_0$ for any $\omega\in\mathcal{D}_m$. Let us recall an important result of perturbation theory which is a consequence of Theorem 1.10 in \cite{Kato1980Perturbation}:
\begin{thm}[Theorem 1.10 in \cite{Kato1980Perturbation}]\label{ana}
Let $I\subset \mathbb{R}$ and $I\ni z\mapsto M(z)$ be a holomorphic curve of Hermitian  matrices. Then all the eigenvalues and associated eigenvectors of $M(z)$ can be parameterized holomorphically on $I$.
\end{thm}

The following result holds.

\begin{cor}\label{cor}
 Let $\mathcal{D}\supset\mathcal{D'}\ni\omega\mapsto N(\omega)\in\mathcal{M}_H$ be a $C^1$ mapping  verifying
\begin{equation}\label{n}
 \|\mu_i(\omega)-v_i\|_{\mathcal{D}} ,\ \|\partial_{\omega_l}\mu_i(\omega)\|_{\mathcal{D}}<\frac{\min(1,v_0)}{\max{(8n,2d)}},\ \ l=1,\cdots,n,\ \ i=1,\cdots,d,
\end{equation}
where $\mu_i(\omega),i=1,\cdots,d$ are the eigenvalues of $N(\omega)$. Then the results of Proposition \ref{prop} hold for $N_0$ being replaced by $N(\omega)$.
\end{cor}

\begin{proof}
This proof just gives some main points which are different from Proposition \ref{prop}.
Assume that $N(\omega)$ depends analytically  of $\omega$.
For the  equation (\ref{homo1}) with $N_0$ being replaced by $N(\omega)$, by repeating the process of Lemmata \ref{sumlambda} and \ref{pro1}, one has
\begin{equation*}
\lambda_{kij}(\omega)=\langle k,\omega\rangle-\mu_i(\omega)+\mu_j(\omega),\ \ \widetilde{\lambda}_{kiji'j'}(\omega)=\langle k,\omega\rangle \pm\mu_i(\omega)\pm\mu_j(\omega)\pm\mu_{i'}(\omega)\pm\mu_{j'}(\omega),
\end{equation*}
where $\mu_j\in\sigma( N(\omega))$, $j=1,\cdots,d$. Choose $z_{k}=(\text{sign}(k_1),\text{sign}(k_2),\cdots,\text{sign}(k_n))$. From (\ref{n}),  one gets for $k\neq0$,
\begin{equation}\label{lambda}
  \partial_{z_k}\lambda_k(\omega)=\frac{1}{|z_k|_2}|k|-\frac{1}{|z_k|_2}\sum_{l=1}^n \text{sign}(k_l) \partial_{\omega_l}(\mu_i-\mu_j)\geq \frac{1}{\sqrt{n}}\left(|k|-\frac{1}{4}\right) \geq \frac{|k|}{2\sqrt{n}},
\end{equation}
\begin{equation}\label{lambda'}
  \partial_{z_k}\widetilde{\lambda}_k(\omega)=\frac{1}{|z_k|_2}|k|\pm\frac{1}{|z_k|_2}\sum_{l=1}^n \text{sign}(k_l) \partial_{\omega_l}(\mu_i\pm\mu_j\pm\mu_{i'}\pm\mu_{j'})  \geq \frac{|k|}{2\sqrt{n}},
\end{equation}
and then the estimation (\ref{meas}) still holds for  $N(\omega)$ being analytic and  satisfying (\ref{n}).

Now let $N(\omega)$ be only a $C^1$-function of $\omega$. For the fast decreasing sequence $s_v=\frac{1}{2}\sigma_v, v\geq 0$ defined in (\ref{sv}) and from the discussion of subsection \ref{appro},  there exists a sequence of real analytic functions $N^{(v)}, v\geq0$ , such that $N(\omega)=\lim_{v\rightarrow \infty}N^{(v)}(\omega)$ for any $\omega\in \mathcal{D}$ with $N^{(v)}$ being real analytic on $\mathcal{D}'_{s_v}:=\{\omega\in\mathbb{C}^n:\Re\omega\in\mathcal{D}', |\Im\omega|\leq s_v\}$.
Denote $L_k^{(v)}(\omega)=I_{d}\otimes(\langle k,\omega\rangle I_d-N^{(v)}(\omega))+(N^{(v)}(\omega))^{\top}\otimes I_d$ and $\lambda_k^{(v)}(\omega)$ the corresponding eigenvalues. Then $\max|\lambda_k(\omega)|=\lim_{v\rightarrow\infty}\max|\lambda^{(v)}_k(\omega)|$. From (\ref{n}), the estimates (\ref{lambda}) and (\ref{lambda'}) hold true uniformly for the sequence $\lambda^{(v)}(\omega),\ v\geq0$, and (\ref{meas}) remains valid uniformly for $L_k^{(v)}(\omega),\ v\geq0$.  Therefore the measure estimation (\ref{meas}) still holds  for the  $C^1$-function $N(\omega)$.

When considering $\partial_{\omega_j}\hat{F}^{z\bar{z}}_0(k)$, the  equation (\ref{lkomega}) is replaced by
\begin{equation*}
(\langle k,\omega\rangle I_d-N_0)\partial_{\omega_l}\hat{F}_0^{z\bar{z}}(k)+\partial_{\omega_l}\hat{F}_0^{z\bar{z}}(k)N_0
=-\text{i}\partial_{\omega_j}\hat{Q}_{0,0}^{z\bar{z}}(k)-k_j\hat{F}_0^{z\bar{z}}(k) +[\partial_{\omega_j}N(\omega),\hat{F}^{z\bar{z}}_0(k)]
\end{equation*}
and the terms $-k_j\hat{F}_0^{z\bar{z}}(k) +[\partial_{\omega_j}N(\omega),\hat{F}^{z\bar{z}}_0(k)]$ can be controlled by $2|k||L_k^{-1}(\omega)||\hat{Q}_{0,0}^{z\bar{z}}(k)|$ from the condition (\ref{n}). Hence, the estimates (\ref{F01}) and (\ref{F01'}) still hold for $N(\omega)$.

When turning to the equations (\ref{homo3}) and (\ref{homo3'}) for $N(\omega)$, one faces the  small divisors  $|\langle k,\omega\rangle\mp\mu_i(\omega)|$. For $k=0$, these reduce to $|\mu_i(\omega)|$. From (\ref{n}), $|\mu_i(\omega)-v_i|\leq \frac{v_i}{4}$ and hence $|\mu_i(\omega)|\geq \frac{v_0}{2}$.

For the equations (\ref{homo2}) and (\ref{homo2'}), the small divisors  are $\min_{i,j}|\langle k,\omega\rangle\mp (\mu_i(\omega)+\mu_j(\omega))|$. When $k=0$, from (\ref{n}) one gets  $|\mu_i(\omega)-v_i|\leq \frac{v_0}{4}$ and hence $|\mu_i(\omega)+\mu_j(\omega)|\geq |v_i+v_j|-|v_i-\mu_i+v_j-\mu_j|   \geq v_0$.
The rest of the proof are similar to that of  Proposition \ref{prop}.
\end{proof}

\subsection{Coordinates transformation and Estimating the new error terms}\label{sectionerror}
For the time dependent Hamiltonian $\epsilon_0f_0$ defined in Proposition \ref{prop},  let us denote $\mathbf{X}_{\epsilon_0 f_0}$ the Hamiltonian vectorfield associated with the equations of motion
\begin{equation*}
\dot{z}=-\text{i}\epsilon_0\frac{\partial f_0}{\partial\bar{z}},\ \ \dot{z}=\text{i}\epsilon_0\frac{\partial f_0}{\partial z},\ \ \dot{\theta}=\omega.
\end{equation*}
It is known from the discussions in subsection \ref{sectionhomo} that the time-one flow $\Phi_0=\textbf{X}_{\epsilon_0 f_0}^\kappa|_{\kappa=1}$ transforms the Hamiltonian $h_0+\epsilon_0q_{0,0}$ into
\begin{align}
\nonumber(h_0+\epsilon_0q_{0,0})\circ \Phi_0=&\ h_0+\langle z,\widetilde{N}_1(\omega)\bar{z}\rangle+e_1\\
\label{error0} &+r_1\\
 \label{error1}  &+\epsilon_0^2\int_0^1(1-t)\{\{h_0,f_0\}-\omega\cdot\partial_{\theta}f_0,f_0\}\circ\textbf{X}_{\epsilon_0f_0}^{\kappa}\text{d}\kappa \\
 \label{error2} &+\epsilon_0^2\int_0^1\{q'_{0},f_0\}\circ\textbf{X}_{\epsilon_0f_0}^{\kappa}\text{d}\kappa,
\end{align}
where $f_0, e_1=\tilde{e}_1, r_1, N_1=N_0+\widetilde{N}_1$ have the properties in Proposition \ref{prop}.
Denote
$$h_1=h_0+\langle z,\widetilde{N}_1(\omega)\bar{z}\rangle+e_1,\ \ \ \  \epsilon_1q_1=(\ref{error0})+(\ref{error1})+(\ref{error2}),$$
 \begin{equation}\label{q1'}
  q'_1=q_1+q_{0,1}\circ\Phi_0.
\end{equation}

In the following,  the coordinates transformation $\Phi_0$ and the new error terms (\ref{error1})-(\ref{error2}) are considered, so that the new Hamiltonian has the form $h_1+\epsilon_1q_{1}$.
\begin{itemize}
  \item Coordinates transformation.
\end{itemize}

it is  claimed  that  for $\omega\in\mathcal{D}_1$, the mapping
\begin{align}\label{flow}
\nonumber\Phi_0=\mathbf{X}_{\epsilon_0 f_0}^\kappa|_{\kappa=1}:\mathbb{T}^n_{s_1}\times \mathbb{C}^{2d}&\rightarrow\mathbb{T}^n_{s_0}\times \mathbb{C}^{2d},\ \ \omega\in\mathcal{D}_1,\\
 (\theta,z,\bar{z})&\mapsto(\theta,\phi_0(\theta,z,\bar{z}))
\end{align}
is an affine transformation in $(z,\bar{z})$, analytic in $\theta\in\mathbb{T}^n_{s^1_{1}}$ and  $C^1$ in $\omega\in\mathcal{D}_1$. Moreover, for each $\theta\in\mathbb{T}^n$, $(z,\bar{z})\mapsto \phi_0(\theta,z,\bar{z})$ is a symplectic map of variables in $\mathbb{C}^{2d}$.

In fact, for $\omega\in\mathcal{D}_1$ and $(\theta,z,\bar{z})\in\mathbb{T}^n_{s^1_1}\times \mathbb{C}^{2d}$ with $z$ and $\bar{z}$ being a pair of conjugate vectors, one can denote $u(t)=\phi_0^t(\theta,z,\bar{z})$  the integral curve of the Hamiltonian vector field  $\mathbf{X}_{\epsilon_0 f_0}|_{\mathbb{C}^{2d}}$ with the initial value
 $u_0=\left(\begin{array}{c}
  z \\
  \bar{z}
\end{array}\right)\in \mathbb{C}^{2d}$. Then
\begin{equation*}
  u(\kappa)=u_0+\epsilon_0\int_{0}^\kappa (A(\theta;\omega)u(s)+b(\theta;\omega))\text{d}s,
\end{equation*}
where
\begin{equation}
A(\theta;\omega)=\left(\begin{array}{cc}
                         -\text{i}(F_0^{z\bar{z}}(\theta))^{\top} & -\text{i}({F}_0^{\bar z\bar z}(\theta)+({F}_0^{\bar z\bar z}(\theta))^{\top}) \\
                         \text{i}(F_0^{zz}(\theta)+(F_0^{zz}(\theta))^{\top}) & \text{i}F_0^{z\bar{z}}(\theta)
                       \end{array}\right),\ \
 b(\theta;\omega)=\left(\begin{array}{c}
                                           -\text{i}{F}_0^{\bar{z}}(\theta)\\
                                           \text{i}F_0^z(\theta)
                            \end{array}\right)
\end{equation}
with
\begin{equation}\label{ab}
  \|A(\theta;\omega)\|_{s^1_1,\mathcal{D}_1}, \ \|b(\theta;\omega)\|_{s^1_1,\mathcal{D}_1} \leq 2[f_0]^0_{s^1_1,\mathcal{D}_1}.
\end{equation}
Here, $\|A\|:=\max\{\|A_j\|:j=1,\cdots,4\}$ for the matrix $A=\left(\begin{array}{cc}
A_1&A_2\\A_3&A_4 \end{array}\right)$ with $A_j$ being $d\times d$ matrices and $\|b\|:=\max\{|b_1|,|b_2|\}$ with $b_1$ and $b_2$ being vectors in $\mathbb{C}^d$.
For $\kappa\in[0,1]$ and  $\|u_0\|\leq 1$,
\begin{equation*}
  \|u(\kappa)-u_0\|\leq \epsilon_0\int_{0}^\kappa \|A(\theta;\omega)\|\|(u(s)-u_0)\|\text{d}s+\epsilon_0 \kappa (\|(b(\theta)\|+\|A(\theta)\|),
\end{equation*}
and from  Gronwall's inequality and (\ref{ab}),
\begin{align*}
  \|u(t)-u_0\|\leq& \epsilon_0(\|b(\theta)\|+\|A(\theta)\|)e^{\epsilon_0\int_{0}^\kappa \|A(\theta;\omega)\|\text{d}s}\\
  \leq& 4\epsilon_0 [f_0]^{(0)}_{s^1_1,\mathcal{D}_1}e^{\kappa\epsilon_0 [f_0]^{(0)}_{s^1_1,\mathcal{D}_1}}, \ \ t\in [0,1],
\end{align*}
which combining (\ref{ell}) and (\ref{f01}) implies
\begin{equation}\label{phi}
  \|\phi_0-id\|_{\mathbb{C}^{2d}}\leq C(1)\epsilon_0^{1-\frac{(1+\rho)\tau}{\ell}}
\end{equation}
for all $\theta\in\mathbb{T}^n_{s_2}$ and $\omega\in\mathcal{D}_1$. Therefore, one gets the estimation of the coordinates transformation
\begin{equation}\label{phi-id}
  \|\Phi_0-id\|_{\mathbb{T}^n_{s^2_1}\times\mathbb{C}^{2d}}\leq C(1)\epsilon_0^{1-\frac{(1+\rho)\tau}{\ell}}
  \leq C(1)\epsilon_0^{\frac{1-\rho}{2}},\ \ \omega\in\mathcal{D}_1.
\end{equation}
By Gronwall's inequality, (\ref{ell}) and the choice of $\epsilon_0$,
\begin{equation}\label{phi-id}
  \|\partial_{\omega_l}(\Phi_0-id)\|_{\mathbb{T}^n_{s^2_1}\times\mathbb{C}^{2d}}\leq C(1)\epsilon_0^{1-\frac{(1+\rho)(2\tau+1)}{\ell}}
  \leq C(1)\epsilon_0^{\rho},\ \  l=1,2,\cdots,n,\ \ \omega\in\mathcal{D}_1.
\end{equation}
Moreover, from Cauchy estimates,
\begin{equation}\label{partialphi0}
  \|\partial(\Phi_0-id)\|_{\mathbb{T}^n_{s^3_1}\times\mathbb{C}^{2d}}
  =\|\partial\Phi_0\|_{\mathbb{T}^n_{s^3_1}\times\mathbb{C}^{2d}}
  \leq C(1)\epsilon_0^{1-\frac{(1+\rho)(1+\tau)}{\ell}}\leq C(1)\epsilon_0^{\frac{1+\rho}{4}},
\end{equation}
which implies $\|\partial\Phi_0\|_{\mathbb{T}^n_{s_{1}}\times\mathbb{C}^{2d}}\leq 2$.
It follows immediately that
\begin{equation}\label{tildephi}
\Phi_0(\mathbb{T}^n_{s_1}\times\mathbb{C}^{2d})\subset\mathbb{T}^n_{\sigma_1}\times\mathbb{C}^{2d}.
\end{equation}
Therefore,
\begin{equation}\label{Phi0}
 \Phi_0:\mathbb{T}^n_{s_1}\times \mathbb{C}^{2d}\rightarrow\mathbb{T}^n_{s_0}\times \mathbb{C}^{2d},\ \ \ \ \ \Phi_0(\mathbb{T}^n_{s_1}\times\mathbb{C}^{2d})\subset\mathbb{T}^n_{\sigma_1}\times\mathbb{C}^{2d},\ \ \ \ \ \omega\in\mathcal{D}_1,
\end{equation}
with the estimates (\ref{phi-id})-(\ref{partialphi0}).

\begin{itemize}
  \item Estimation of (\ref{error2}).
\end{itemize}

For a function $r\in\mathcal{Q}_{s,\mathcal{D}}$,  denote
\begin{equation*}
  r(\theta,z,\bar{z})=\langle u,R^1(\theta)u\rangle+\langle R^2(\theta),u\rangle + R^3(\theta),
\end{equation*}
where
\begin{equation*}
   R^1(\theta)=\left(
                   \begin{array}{cc}
                     R^{zz}(\theta) & \frac{1}{2}R^{z\bar{z}} (\theta)\\
                     \frac{1}{2}(R^{z\bar{z}}(\theta))^{\top} & {R}^{\bar z\bar z}(\theta) \\
                   \end{array}
                 \right),
   R^2(\theta)=\left(
                          \begin{array}{c}
                            R^{z}(\theta) \\
                            {R}^{\bar z}(\theta) \\
                          \end{array}
                        \right),
   R^3(\theta)=R^{\theta}(\theta).
\end{equation*}
 Then the functions $f_0,q_{0,0}$ defined in Proposition \ref{prop} and (\ref{qv}) have the forms
\begin{equation}\label{f02}
  f_{0}(\theta,z,\bar{z})=\langle u,f_0^1(\theta)u\rangle+\langle f_0^2(\theta),u\rangle+f_0^3(\theta),
\end{equation}
\begin{equation}\label{q00}
  q_{0,0}(\theta,z,\bar{z})=\langle u,Q_{0,0}^1(\theta)u\rangle+\langle Q_{0,0}^2(\theta),u\rangle+Q_{0,0}^3(\theta).
\end{equation}
From  (\ref{con1}) and  Proposition \ref{prop}, one has the estimations
\begin{equation}\label{f012}
  \|f_0^1\|_{s^{1}_1,\mathcal{D}_1}, \|f_0^2\|_{s^1_1,\mathcal{D}_1},\|f_0^3\|_{s^1_1,\mathcal{D}_1}\leq C(1)\epsilon_0^{-\frac{(1+\rho)\tau}{\ell}},
\end{equation}
\begin{equation}\label{pf012}
  \|\partial_{\omega_l}f_0^1\|_{s^1_1,\mathcal{D}_1}, \|\partial_{\omega_l}f_0^2\|_{s^1_1,\mathcal{D}_1}, \|\partial_{\omega_l}f_0^3\|_{s^1_1,\mathcal{D}_1}\leq C(1)\epsilon_0^{-\frac{(1+\rho)(2\tau+1)}{\ell}},\ \ l=1,\cdots,n,
\end{equation}
\begin{equation}\label{q012}
   \|Q_{0,0}^j\|_{s_0,\mathcal{D}}, \|\partial_{\omega_l}Q_{0,0}^j\|_{s_0,\mathcal{D}}\leq C(0),\ \ j=1,2,3,\ l=1,\cdots,n.
\end{equation}
Let $S=\left(
           \begin{array}{cc}
             0 & -\text{i}I_d \\
             \text{i}I_d & 0 \\
           \end{array}
         \right)
$. From the definition of Poisson bracket in (\ref{possion}), it appears
\begin{equation*}
  \{q_{0,0},f_0\}=\langle\frac{\partial q_{0,0}}{\partial u},S\frac{\partial f_0}{\partial u}\rangle=4\langle u,\tilde{Q}_{0,0}^{1}u\rangle +4\langle \tilde{Q}_{0,0}^{2},u\rangle +\tilde{Q}_{0,0}^3,
\end{equation*}
where
\begin{equation*}
\tilde{Q}_{0,0}^{1}=Q_{0,0}^1Sf_0^1,\ \ \tilde{Q}_{0,0}^{2}=\frac{1}{2}(Q_{0,0}^1Sf_0^2-f_0^1SQ_{0,0}^2),\ \ \tilde{Q}_{0,0}^{3}=\langle Q_{0,0}^2,Sf_0^2\rangle.
\end{equation*}
Noting that the Poisson bracket of $q_{0,0}$ and $f_0$ does not contain the derivatives in $\theta$,   it is easy to check that $\{q_{0,0},f_0\}\in \mathcal{Q}^{\Re}_{s_1,\mathcal{D}_1}$ since $q_{0,0}\in \mathcal{Q}^{\Re}_{\sigma_0,\mathcal{D}_1}$ and $f_0\in \mathcal{Q}^{\Im}_{s_1,\mathcal{D}_1}$. From (\ref{tildephi}), $\{q_{0,0},f_0\}\circ\textbf{X}_{\epsilon_0f_0}^{\kappa}$ makes sense and it can be computed by Taylor's formula
\begin{equation*}
   \{q_{0,0},f_0\}\circ\textbf{X}_{\epsilon_0f_0}^{\kappa}=\{q_{0,0},f_0\}+\kappa\epsilon_0\{\{q_{0,0},f_0\},f_0\} +\frac{1}{2!}\kappa^2\epsilon_0^2\{\{\{q_{0,0},f_0\},f_0\},f_0\}+\cdots.
\end{equation*}
 Denote
\begin{equation}\label{e2}
  \{q_{0,0},f_0\}\circ\textbf{X}_{\epsilon_0f_0}^{\kappa}=\langle u,\tilde{Q}_{0,0}^{1\star}(\theta)u\rangle+\langle \tilde{Q}_{0,0}^{2\star}(\theta),u\rangle +\tilde{Q}_{0,0}^{3\star}(\theta).
\end{equation}
Then $\tilde{Q}_{0,0}^{3\star}=\tilde{Q}_{0,0}^3$  with
 \begin{equation*}
 |\tilde{Q}_{0,0}^{3\star}\|_{s_1,\mathcal{D}_1}\leq C(1)\epsilon_0^{-\frac{(1+\rho)\tau}{\ell}},\ \
 \|\partial_{\omega_l}\tilde{Q}_{0,0}^{3\star}\|\leq C(1)\epsilon_0^{-\frac{(1+\rho)(2\tau+1)}{\ell}},\ l=1,\cdots,n
\end{equation*}
and
\begin{equation*}
  \tilde{Q}_{0,0}^{1\star}=4\tilde{Q}_{0,0}^{1}+4^2\kappa\epsilon_0\tilde{Q}_{0,0}^{1}Sf_0^1+\cdots +\frac{4^{j+1}\kappa^j\epsilon_0^j}{j!}\tilde{Q}_{0,0}^{1}(Sf_0^1)^j+\cdots.
\end{equation*}
From (\ref{f012})-(\ref{q012}), one has the estimations
\begin{align*}
  \|\tilde{Q}_{0,0}^{1\star}\|_{s_1,\mathcal{D}_1}&\leq 4\|\tilde{Q}_{0,0}^{1}\|_{s_1,\mathcal{D}_1}\sum_{j=0}^{\infty}\frac{(4\kappa\epsilon_0)^j}{j!}\|f_0^1\|^j_{s_1,\mathcal{D}_1}\\
  &\leq 4[q_{0,0}]_{s_0,\mathcal{D}}[f_0]^{(0)}_{s_1,\mathcal{D}_1}e^{4\kappa\epsilon_0[f_0]^{0}_{s_1,\mathcal{D}_1}}\\
  &\leq C(1)\epsilon_0^{-\frac{(1+\rho)\tau}{\ell}},
\end{align*}
\begin{align*}
  \|\partial_{\omega_l}\tilde{Q}_{0,0}^{1\star}\|_{s_1,\mathcal{D}_1}&\leq 4\|\partial_{\omega_l}({Q}_{0,0}^{1}Sf_0^1)\|_{s_1,\mathcal{D}_1} +4\sum_{j=1}^{\infty}\frac{(4\kappa\epsilon_0)^j}{j!}\|\partial_{\omega_l}({Q}_{0,0}^{1}(Sf_0^1)^{j+1})\|_{s_1,\mathcal{D}_1}\\
  &\leq 4(\|\partial_{\omega_l}{Q}_{0,0}^{1}\|_{s_0,\mathcal{D}} \|f_0^1\|_{s_1,\mathcal{D}_1} +\|{Q}_{0,0}^{1}\|_{s_0,\mathcal{D}}\|\partial_{\omega_l}f_0^1\|_{s_1,\mathcal{D}_1} )\\
  &\ \ + 4\sum_{j=1}^{\infty}\frac{(4\kappa\epsilon_0)^j}{j!}\| f_0^1\|_{s_1,\mathcal{D}_1}^{j} \left[\|\partial_{\omega_l}{Q}_{0,0}^{1}\|_{s_0,\mathcal{D}}\|f_0^1\|_{s_1,\mathcal{D}_1}
  +(j+1)\|{Q}_{0,0}^{1}\|_{s_0,\mathcal{D}} \|\partial_{\omega_l}f_0^1\|_{s_1,\mathcal{D}_1}\right]\\
  &\leq 8C(1)\epsilon_0^{-\frac{(1+\rho)(2\tau+1)}{\ell}}
  + 3\times 4^2\kappa\epsilon_0 C(1)\epsilon_0^{-\frac{(1+\rho)(2\tau+1)}{\ell}} \sum_{j=0}^{\infty}\frac{(4\kappa\epsilon_0)^j}{j!}\|f_0^1\|^j_{s_1,\mathcal{D}_1}\\
  &\leq C(1)\epsilon_0^{-\frac{(1+\rho)(2\tau+1)}{\ell}}, \ \ 1\leq l\leq n.
\end{align*}
Similarly,
\begin{equation*}
  \|\tilde{Q}_{0,0}^{2\star}\|_{{s_1,\mathcal{D}_1}}\leq C(1)\epsilon_0^{-\frac{(1+\rho)\tau}{\ell}},\  \|\partial_{\omega_l}\tilde{Q}_{0,0}^{2\star}\|_{s_1,\mathcal{D}_1}\leq C(1)\epsilon_0^{-\frac{(1+\rho)(2\tau+1)}{\ell}},\ \ l=1,\cdots,n.
\end{equation*}
Together with (\ref{ell}) and (\ref{e2})  ,
\begin{equation*}
  \left[\epsilon_0^2\int_0^1\{q_{0,0},f_0\}\circ\textbf{X}_{\epsilon_0f_0}^{\kappa}\text{d}\kappa\right]^{(0)}_{s_1,\mathcal{D}_1}\leq C(1)\epsilon_0^{2-\frac{(1+\rho)\tau}{\ell}}\leq C(1)\epsilon_0^{\frac{5+\rho}{4}},
\end{equation*}
\begin{equation*}
  \left[\epsilon_0^2\int_0^1\{q_{0,0},f_0\}\circ\textbf{X}_{\epsilon_0f_0}^{\kappa}\text{d}\kappa\right]^{(1)}_{s_1,\mathcal{D}_1}\leq C(1)\epsilon_0^{2-\frac{(1+\rho)(2\tau+1)}{\ell}}\leq C(1)\epsilon_0^{1+\rho},
\end{equation*}
and therefore the term (\ref{error2}) belongs to $\mathcal{Q}^{\Re}_{s_1,\mathcal{D}_1}$  with the estimation
\begin{equation}\label{error2'}
  \left[\epsilon_0^2\int_0^1\{q_{0,0},f_0\}\circ\textbf{X}_{\epsilon_0f_0}^{\kappa}\text{d}\kappa\right]_{s_1,\mathcal{D}_1}\leq C(1)\epsilon_1.
\end{equation}

\begin{itemize}
  \item Estimation of (\ref{error1}).
\end{itemize}
From the homological equation (\ref{1}), one has
\begin{equation*}
  \epsilon_0\{h_0,f_0\}-\epsilon_0\omega\cdot\partial_{\theta}f_0=\langle z,\widetilde{N}_{1}\bar z\rangle+r_{1}+\tilde{e}_1-\epsilon_0q_{0,0}.
\end{equation*}
Recalling (\ref{f02}), and similarly to (\ref{q00}), it appears
\begin{equation*}
  \tilde{h}_0(\theta,z,\bar{z}):=\epsilon_0\{h_0,f_0\}-\epsilon_0\omega\cdot\partial_{\theta}f_0=\langle u,\tilde{h}_0^{1,0}(\theta)u\rangle+\langle\tilde{h}_0^{2,0}(\theta),u\rangle +\tilde{h}_0^{3,0}(\theta),
\end{equation*}
with \begin{equation*}\tilde{h}_0^{1,0}(\theta)=\left(
        \begin{array}{cc}
          R_1^{zz}(\theta)-\epsilon_0Q_{0,0}^{zz}(\theta) & \frac{1}{2}(\widetilde{N}_1(\theta)+R_1^{z\bar{z}}(\theta)-\epsilon_0Q_{0,0}^{z\bar{z}}(\theta)) \\
          \frac{1}{2}(\widetilde{N}_1(\theta)+R_1^{z\bar{z}}(\theta)-\epsilon_0Q_{0,0}^{z\bar{z}}(\theta))  & \bar{R}_1^{zz}(\theta)-\epsilon_0\bar{Q}_{0,0}^{zz}(\theta)
        \end{array}
      \right),
\end{equation*}
\begin{equation*}\tilde{h}_0^{2,0}(\theta)=\left(
                               \begin{array}{c}
                                 R_1^z(\theta)-\epsilon_0Q_{0,0}^z(\theta) \\
                                 \bar{R}_1^z(\theta)-\epsilon_0\bar{Q}_{0,0}^z(\theta) \\
                               \end{array}
                             \right),\ \ \tilde{h}_0^{3,0}(\theta)=\epsilon_0\hat{Q}^{\theta}_{0,0}(0)+R_1^{\theta}(\theta)-\epsilon_0{Q}_{0,0}^{\theta}(\theta) .
\end{equation*}
 Then
 \begin{equation*}
   \{\tilde{h}_0,f_0\}=4\langle u,\tilde{h}_{0}^{1,1}u\rangle+4\langle \tilde{h}_0^{2,1},u\rangle+\tilde{h}_0^{3,1},
 \end{equation*}
where
\begin{equation*}
  \tilde{h}_0^{1,1}=\tilde{h}_0^{1,0}Sf_0^1,\ \  \tilde{h}_0^{2,1}= \frac{1}{2}(\tilde{h}_0^{1,0}Sf_0^2-f_0^1S\tilde{h}_0^{2,0}),\ \ \tilde{h}_0^{3,1}=\langle \tilde{h}_0^{2,0},Sf_0^{2}\rangle,
\end{equation*}
and $f_0^1,\ f_0^2$ are defined in (\ref{f02}).
Since $q_{0,0},r_0\in \mathcal{Q}^{\Re}_{s_1,\mathcal{D}_1}$ and $f_0\in \mathcal{Q}^{\Im}_{s_1,\mathcal{D}_1}$,  it appears that $\{\tilde{h}_0,f_0\}\in \mathcal{Q}^{\Re}_{s_1,\mathcal{D}_1}$. Therefore $ \{\tilde{h}_0,f_0\}\circ\textbf{X}_{\epsilon_0f_0}^{\kappa}$ is reasonable by (\ref{tildephi}) and can be computed by Taylor's formula
\begin{equation*}
   \{\tilde{h}_0,f_0\}\circ\textbf{X}_{\epsilon_0f_0}^{\kappa}=\{\tilde{h}_0,f_0\}+\kappa\epsilon_0\{\{\tilde{h}_0,f_0\},f_0\} +\frac{1}{2!}\kappa^2\epsilon_0^2\{\{\{\tilde{h}_0,f_0\},f_0\},f_0\}+\cdots.
\end{equation*}
Denote
\begin{equation}\label{e1}
  \{\tilde{h}_0,f_0\}\circ\textbf{X}_{\epsilon_0f_0}^{\kappa}=\langle u,\tilde{h}_{0}^{1\star}(\theta)u\rangle+\langle \tilde{h}_{0}^{2\star}(\theta),u\rangle +\tilde{h}_{0}^{3\star}(\theta).
\end{equation}
One has  $\tilde{h}_{0}^{3\star}=\tilde{h}_{0}^{3,1}$ with
\begin{equation*}
  \|\tilde{h}_{0}^{3\star}\|_{s_1,\mathcal{D}_1}\leq C(1)\epsilon_0^{1-\frac{(1+\rho)\tau}{\ell}},\ \
  \|\partial_{\omega_l}\tilde{h}_{0}^{3\star}\|_{s_1,\mathcal{D}_1}\leq C(1)\epsilon_0^{1-\frac{(1+\rho)(2\tau+1)}{\ell}},\ \ l=1,\cdots,n,
\end{equation*}
and
\begin{equation*}
  \tilde{h}_{0}^{1\star}=\sum_{j=1}^{\infty}4^j\tilde{h}_0^{1,j}\frac{(\kappa\epsilon_0)^{j-1}}{(j-1)!},\ \ \ \ \ \ \
  \tilde{h}_{0}^{2\star}= \sum_{j=1}^{\infty}4^j\tilde{h}_0^{2,j}\frac{(\kappa\epsilon_0)^{j-1}}{(j-1)!},
\end{equation*}
where
\begin{equation*}
\tilde{h}_{0}^{1,j}=\tilde{h}_{0}^{1,j-1}Sf_0^1=\tilde{h}_{0}^{1,0}(Sf_0^1)^j,\ \ \ \ \ \ \tilde{h}_0^{2,j}=\frac{1}{2}(\tilde{h}_0^{1,j-1}Sf_0^2-f_0^1S\tilde{h}_0^{2,j-1}),\ \ j\geq 1.
\end{equation*}
From Proposition \ref{prop},
\begin{align*}
  \|\tilde{h}_{0}^{1,j}\|_{s_1,\mathcal{D}_1}\leq \|\tilde{h}_0^{1,0}\|_{s_1,\mathcal{D}_1}\|f_0^1\|_{s_1,\mathcal{D}_1}^{j}
  \leq C(1)\epsilon_0\left([f_0]_{s_1,\mathcal{D}_1}^{(0)}\right)^{j},
\end{align*}
\begin{align*}
  \|\partial_{\omega_l}\tilde{h}_{0}^{1,j}\|_{s_1,\mathcal{D}_1}
  &\leq\|\partial_{\omega_l}\tilde{h}_{0}^{1,0}\|_{s_1,\mathcal{D}_1}\|f_0^1\|^j_{s_1,\mathcal{D}_1}
  +j\|\tilde{h}_{0}^{1,0}\|_{s_1,\mathcal{D}_1}\|f_0^1\|_{s_1,\mathcal{D}_1}^{j-1}\|\partial_{\omega_l}f_0^1\|_{s_1,\mathcal{D}_1} \\
  &\leq C(0)\left([f_0]_{s_1,\mathcal{D}_1}^{(0)}\right)^{j-1} \left([f_0]^{(0)}_{s_1,\mathcal{D}_1}+j[f_0]^{(1)}_{s_1,\mathcal{D}_1}\right)\\
  &\leq (j+1)C(1)\epsilon_{0}^{1-\frac{(1+\rho)(2\tau+1)}{\ell}}\left([f_0]^{(0)}_{s_1,\mathcal{D}_1}\right)^{j-1},\ \ 1\leq l\leq n.
\end{align*}
Therefore, one has the estimates
\begin{equation*}
  \|\tilde{h}_{0}^{1\star}\|_{s_1,\mathcal{D}_1}
  \leq 4C(1)\epsilon_0[f_0]^{(0)}_{s_1,\mathcal{D}_1}\sum_{j=0}^{\infty}\frac{(4\kappa\epsilon_0)^j}{j!}\left([f_0]^{(0)}_{s_1.\mathcal{D}_1}\right)^{j}
 \leq C(1)\epsilon_{0}^{1-\frac{(1+\rho)\tau}{\ell}}
\end{equation*}
\begin{equation*}
  \|\partial_{\omega_l}\tilde{h}_{0}^{1\star}\|_{s_1,\mathcal{D}_1}
  \leq \|\partial_{\omega_l}\tilde{h}_0^{1,1}\|_{s_1,\mathcal{D}_1} +\sum_{j=2}^{\infty}4^j\frac{(\kappa\epsilon_0)^{j-1}}{(j-1)!}\|\partial_{\omega_l}\tilde{h}_0^{1,j}\|_{s_1,\mathcal{D}_1}
  \leq C(1)\epsilon_0^{1-\frac{(1+\rho)(2\tau+1)}{\ell}}.
\end{equation*}
Similarly,
\begin{equation*}
   \|\tilde{h}_{0}^{2\star}\|_{s_1,\mathcal{D}_1}\leq C(1)\epsilon_0^{1-\frac{(1+\rho)\tau}{\ell}},\ \ \ \ \ \ \|\partial{\omega_l}\tilde{h}_{0}^{2\star}\|_{s_1,\mathcal{D}_1}\leq C(1)\epsilon_0^{1-\frac{(1+\rho)(2\tau+1)}{\ell}},\ \ l=1,\cdots,n.
\end{equation*}
Together with  (\ref{ell}) and (\ref{e1}), one has the estimates of (\ref{error1})
\begin{equation*}
  \left[\epsilon_0^2\int_0^1(1-\kappa)\{\{h_0,f_0\}-\omega\cdot\partial_{\theta}f_0,f_0\}\circ\textbf{X}_{\epsilon_0f_0}^{\kappa}\text{d}\kappa \right]^{(0)}_{s_1,\mathcal{D}_1}
  \leq C(1)\epsilon_0^{2-\frac{(1+\rho)\tau}{\ell}} \leq C(1)\epsilon_0^{\frac{3-\rho}{2}},
\end{equation*}
\begin{equation*}
  \left[\epsilon_0^2\int_0^1(1-\kappa)\{\{h_0,f_0\}-\omega\cdot\partial_{\theta}f_0,f_0\}\circ\textbf{X}_{\epsilon_0f_0}^{\kappa}\text{d}\kappa \right]^{(1)}_{s_1,\mathcal{D}_1}
  \leq C(1)\epsilon_0^{2-\frac{(1+\rho)(2\tau+1)}{\ell}} \leq C(1)\epsilon_0^{1+\rho},
\end{equation*}
and therefore the term (\ref{error1}) belongs to $\mathcal{Q}^{\Re}_{s_1,\mathcal{D}_1}$  with the estimation
\begin{equation}\label{error1'}
  \left[\epsilon_0^2\int_0^1(1-\kappa)\{\{h_0,f_0\}-\omega\cdot\partial_{\theta}f_0,f_0\}\circ\textbf{X}_{\epsilon_0f_0}^{\kappa}\text{d}\kappa \right]_{s_1,\mathcal{D}_1}
  \leq C(1)\epsilon_1.
\end{equation}

\begin{itemize}
  \item Estimation of $q_{0,1}\circ\Phi_0$.
\end{itemize}

Note that  $q_{0,1}$ is analytic in $\mathbb{T}^n_{\sigma_1}$. From (\ref{tildephi}), $q_{0,1}\circ\Phi_0\in\mathcal{Q}^{\Re}_{s_1,\mathcal{D}_1}$. Similar to the estimation of (\ref{error2}), it can be  checked  that
\begin{equation}\label{error3'}
[q_{0,1}\circ\Phi_0]_{s_1,\mathcal{D}_1}\leq C(1).
\end{equation}

\begin{itemize}
  \item The new Hamiltonian.
\end{itemize}
Let
\begin{equation*}
h_1=e_1+\langle z,N_1(\omega)\bar{z}\rangle,
\end{equation*}
where
\begin{equation*}
 e_1(\omega)=e_0+\tilde{e}_1(\omega)=\epsilon_0\hat{Q}_{0}^{'\theta}(0),\ \ \ \ \  N_1(\omega)= N_0+\widetilde{N}_1(\omega),
\end{equation*}
and let
\begin{equation}\label{q11}
\epsilon_1q_{1}=(\ref{error0})+(\ref{error1})+(\ref{error2}), \ \ \ \ \ q'_{1}=q_1+q_{0,1}\circ\Phi_0.
\end{equation}

As a conclusion, for $\omega\in\mathcal{D}_1\subset\mathcal{D}$ with $\meas(\mathcal{D}-\mathcal{D}_1)\leq \gamma_0^{\frac{1}{2}}$, one has the new Hamiltonian defined in $\mathbb{T}^{n}_{s_1}\times\mathbb{C}^{2d}$
\begin{equation}
h_1+\epsilon_1 q_{1}=(h_{0}+\epsilon_0 q'_{0})\circ\Phi_{0},
\end{equation}
where $h_1$ is in normal form with $e_1(\omega)$ being real number,  $N_1(\omega)$ being Hermitian matrix for any $\omega\in\mathcal{D}_1$ and $q'_{1}\in\mathcal{Q}^{\Re}_{s_1,\mathcal{D}_1}$
 with the estimates
\begin{equation*}
\|e_1-e_0\|_{\mathcal{D}_1},\ \|N_1-N_0\|_{\mathcal{D}_1},\  \|\partial_{\omega_l}(N_1-N_0)\|_{\mathcal{D}_1}\leq C(1)\epsilon_0,
\ \ \ \ \  l=1,\cdots,n,
\end{equation*}
\begin{equation*}
[q_{1}]_{s_1,\mathcal{D}_1},\ \ [q'_{1}]_{s_1,\mathcal{D}_1}\leq C(1),
\end{equation*}
\begin{equation*}
   \|\Phi_0-id\|_{\mathbb{T}^{n}_{s_1}\times\mathbb{C}^{2d}}\leq C(1)\epsilon_0^{\frac{1-\rho}{2}},
\end{equation*}
\begin{equation*}
  \|\partial(\Phi_0-id)\|_{\mathbb{T}^{n}_{s_1}\times\mathbb{C}^{2d}}\leq C(1)\epsilon_0^{\frac{1+\rho}{4}}.
\end{equation*}
\begin{equation*}
  \|\partial_{\omega_l}(\Phi_0-id)\|_{\mathbb{T}^{n}_{s_1}\times\mathbb{C}^{2d}}\leq C(1)\epsilon_0^{\rho}\leq \epsilon_0^{\frac{1}{2}},\ \ \ \ \ l=1,2,\cdots,n.
\end{equation*}

\subsection{Iterative lemma}
For $\rho$ and $\gamma_0$ defined in (\ref{rho}) and (\ref{con3}), let  \begin{equation}\label{epsilon}
\epsilon_0=\epsilon,\ \ \ \  \epsilon_m=\epsilon_0^{\left(1+\rho\right)^{m}},\ \ \ \ \ m=1,2\cdots,
\end{equation}
\begin{equation}\label{s}
s_{m-1}=\epsilon_{m}^{\frac{1}{\ell}},\ \ \  \ s^j_m=\frac{1}{4}(js_m+(4-j)s_{m+1}),\ \ \ j=1,2,3, \ \ \  m=1,2,\cdots,
\end{equation}
\begin{equation}\label{k}
  K_m=\frac{\log\epsilon_m^{-1}}{s_m-s^1_{m+1}},\ \  \ \   \gamma_m=\frac{\gamma_0}{2^m},\ \ \ \  m=0,1,\cdots.
\end{equation}

Let $q'_{0}=q_{0,0}$, where $q_{0,v}$ is defined in (\ref{q0}).
Denote
\begin{equation*}
 q_{m}(\theta,z,\bar{z})=\langle z,Q_{m}^{zz}(\theta)z\rangle+\langle z,Q_{m}^{z\bar z}(\theta)\bar z\rangle+\langle \bar z,Q_{m}^{\bar{z}\bar{z}}(\theta)\bar{z}\rangle
                      +\langle Q_{m}^{z}(\theta ),z\rangle+\langle Q_{m}^{\bar{z}}(\theta),\bar{z}\rangle+Q_{m}^{\theta}(\theta),\ \ \ \ m \geq 1.
\end{equation*}

\begin{lem}[Iterative lemma]
There exists $\epsilon_{\star}>0$ depending on $d,\ n$ such that for $0<\epsilon<\epsilon_{\star}$
and for all $j\geq1$, there exist closed sets $\mathcal{D}_{j+1}\subset\mathcal{D}_{j}$, Hamiltonians $f_{j}\in\mathcal{Q}^{\Im}_{s_{j+1},\mathcal{D}_{j+1}}$ and
\begin{equation}
h_j=\langle z,N_j(\omega)\bar{z}\rangle+e_j(\omega)
\end{equation}
in normal form where $e_j(\omega)$, $N_j(\omega)$ belong to $\mathbb{C}^1$ on $\mathcal{D}_j$, and
\begin{equation}
q'_j=q_j+q_{0,j}\circ\widetilde{\Phi}_{j-1}
\end{equation}
with $q_{j},\ q'_j\in\mathcal{Q}^{\Re}_{s_j,\mathcal{D}_j}$,
such that for  $j\geq1$ and any $\omega\in \mathcal{D}_j$, the  symplectic change
\begin{equation*}
  \Phi_{j}(\omega)\equiv \textbf{X}_{\epsilon_{j} f_{j}}^\kappa|_{\kappa=1}:\mathbb{T}^n_{s_{j+1}}\times \mathbb{C}^{2d}\rightarrow \mathbb{T}^n_{s_{j}}\times \mathbb{C}^{2d}
\end{equation*}
is an affine transformation in $(z,\bar{z})$, analytic in $\theta\in\mathbb{T}^n_{s_{j+1}}$ and $C^1$ in $\omega\in\mathcal{D}_{j+1}$ of the form
\begin{equation*}
  \Phi_{j}(\theta,z,\bar{z})=(\theta,\phi_{j}(\theta,z,\bar{z})),
\end{equation*}
 where, for each $\theta\in\mathbb{T}^n$, $(z,\bar{z})\mapsto\phi_{j}(\theta,z,\bar{z})$ is a symplectic map of variables on $\mathbb{C}^{2d}$.
The map $\Phi_{j}$ links the Hamiltonian at step $j$ and the Hamiltonian at step $j+1$, i.e.,
\begin{equation*}
 (h_{j}+\epsilon_{j}q'_{j})\circ\Phi_{j}=h_{j+1}+\epsilon_{j+1}q_{j+1},\ \  \forall \omega\in\mathcal{D}_{j+1}.
\end{equation*}
Moreover, one has the following estimates
\begin{equation}\label{measj}
  \meas(\mathcal{D}_{j}-\mathcal{D}_{j+1})\leq \gamma_{j}^{\frac{1}{2}},
\end{equation}
\begin{equation}\label{nj}
\|{e}_{j+1}-e_{j}\|_{\mathcal{D}_{j+1}},\    \| {N}_{j+1}-N_{j}\|_{\mathcal{D}_{j+1}},\ \|\partial_{\omega_l}({N}_{j+1}-N_{j}\|_{\mathcal{D}_{j+1}} \leq C(j+1)\epsilon_{j},\ \ \ \ l=1,\cdots,n,
\end{equation}
\begin{equation}\label{qj}
  [q_{j+1}]_{s_{j+1},\mathcal{D}_{j+1}},\  [q'_{j+1}]_{s_{l+1},\mathcal{D}_{j+1}}\leq C(j+1),
\end{equation}
\begin{equation}\label{phij}
  \|\Phi_{j}-id\|_{\mathbb{T}^n_{s_{j+1}}\times\mathbb{R}^{2d}}\leq C(j+1)\epsilon_{j}^{\frac{1-\rho}{2}}, \ \ \omega\in\mathcal{D}_{j+1}.
\end{equation}
\begin{equation}\label{partialphij}
  \|\partial\Phi_j\|_{\mathbb{T}^n_{s_{j+1}}\times\mathbb{C}^{2d}}
  \leq C(j+1)\epsilon_j^{1-\frac{(1+\rho)(1+\tau)}{\ell}}\leq C(j+1)\epsilon_j^{\frac{1+\rho}{4}},
\end{equation}
\begin{equation}
  \|\partial_{\omega_l}(\Phi_j-id)\|_{\mathbb{T}^n_{s_{j+1}}\times\mathbb{C}^{2d}}\leq C(j+1)\epsilon_j^{\rho},\ \ \ \ l=1,\cdots,n.
\end{equation}
\end{lem}

\begin{proof}
For $j=0$, the lemma is proved in subsections \ref{sectionhomo} and \ref{sectionerror}.

Assume that the lemma is true for $j=m$. Then in the $(m+1)$-step, one aims to find a closed set $\mathcal{D}_{m+1}\subset\mathcal{D}_{m}$, a Hamiltonian $f_{m}\in\mathcal{Q}^{\Im}_{s_{m+1},\mathcal{D}_{m+1}}$ such that after the symplectomorphism $\Phi_m$, one has a new normal form $h_{m+1}$ and the new error terms $\epsilon_{m+1}q_{m+1}$ where the estimations (\ref{measj})-(\ref{phij}) hold true for $j$ being replaced by $m+1$.

Similarly to subsections \ref{sectionhomo}, one should solve the homological equation
\begin{equation}\label{mm}
  \epsilon_m\{h_m,f_m\}+\epsilon_mq'_{m}-\epsilon_m\omega\cdot\partial_{\theta}f_m=\langle z,\widetilde{N}_{m+1}\bar z\rangle+r_{m+1},
\end{equation}
which is equivalent to (\ref{homo1})-(\ref{homo3}) with $0$ being replaced by $m$ and $1$ being replaced by $m+1$.
By considering (\ref{homo1}),  it appears
\begin{equation}\label{em}
  \tilde{e}_m(\omega)=\epsilon_m\hat{Q}_{m}^{'\theta}(0),\  e_m(\omega)=e_{m-1}(\omega)+\tilde{e}_m(\omega),
\end{equation}
\begin{equation}\label{nm}
  \widetilde{N}_{m+1}(\omega)=\epsilon_m\hat{Q}_{m}^{'z\bar{z}}(0),\ \ N_{m+1}(\omega)=N_m(\omega)+\widetilde{N}_{m+1}(\omega),
\end{equation}
 \begin{equation}\label{rm}
 {R}_{m+1}^{z\bar{z}}(\theta)=\epsilon_m\sum_{|k|\geq K_m}\hat{Q}_{m}^{'z\bar{z}}(k)e^{\text{i}\langle k,\theta\rangle},\ \ \omega\in\mathcal{D}_m,
\end{equation}
\begin{equation}\label{homo1m}
(I_{d}\otimes(\langle k,\omega\rangle I_d-N_m)+N_m^{\top}\otimes I_d)(\hat{F}_m^{z\bar{z}}(k))'=-\text{i}(\hat{Q}_{m}^{'z\bar{z}}(k))',\ \ 0<|k|<K_m.
\end{equation}
From (\ref{nm}),  $\mathcal{D}_{m+1}\ni \omega\mapsto N_{m+1}(\omega)\in \mathcal{M}_H$ is a $C^1$ mapping  verifying
\begin{equation*}
 \| N_{m+1}(\omega)-N_0\|,\ \|\partial_{\omega_l}( N_{m+1}(\omega)-N_0)\|\leq \sum_{j=0}^{m}C(m)\epsilon_m\leq C(0)\epsilon_0,\ \ l=1,\cdots,n,
\end{equation*}
From Corollary \ref{cor},  Proposition \ref{prop} holds for the index $0$ and $1$ being replaced by $m$ and $m+1$, i.e., there is a closed subset
 $\mathcal{D}_{m+1}\subset\mathcal{D}_m$ satisfying
\begin{equation*}
  \meas(\mathcal{D}_{m}\setminus\mathcal{D}_{m+1})\leq \gamma_{m}^{\frac{1}{2}},
\end{equation*}
and there exist $f_m\in \mathcal{Q}^{\Im}_{s^1_{m+1},\mathcal{D}_{m+1}}$, $r_{m+1}\in \mathcal{Q}^{\Re}_{s^1_{m+1},\mathcal{D}_{m+1}}$ and $\mathcal{D}_{m+1}\ni\omega\mapsto\widetilde{N}_{m+1}(\omega)\in \mathcal{M}_H$ being a $C^1$ mapping, such that for all $\omega\in\mathcal{D}_{m+1}$, the equation (\ref{mm}) holds.
Furthermore for all $\omega\in\mathcal{D}$,
\begin{equation*}
\|( N_{m+1}(\omega)-N_m(\omega))\|,\ \|\partial_{\omega_l}( N_{m+1}(\omega)-N_m(\omega))\|\leq C(m)\epsilon_m,\ \ l=1,\cdots,n,
\end{equation*}
and
\begin{equation}
   [r_{m+1}]_{s^1_{m+1},\mathcal{D}_{m+1}}\leq 2^{n}\epsilon_m C(m)K_m^{\frac{n}{2}}e^{-K_m(s_m-s^1_{m+1})}\leq C(m+1)\epsilon_m^{\frac{18}{8}},
 \end{equation}
\begin{equation}
  [f_m]^{(0)}_{s^1_{m+1},\mathcal{D}_{m+1}}\leq \frac{2^{\frac{3n}{2}+\frac{3}{2}}\sqrt{(2n-1)!}}{\gamma_m[2(s_m-s^1_{m+1})]^{\tau}}\|Q^{'z\bar{z}}_{m}\|_{s_m,\mathcal{D}_m}
 \leq C(m+1)\epsilon_m^{-\frac{(1+\rho)\tau}{\ell}},
\end{equation}
\begin{equation}\label{pf0m}
  [f_m]^{(1)}_{s^1_{m+1},\mathcal{D}_{m+1}}\leq C(m+1)\epsilon_m^{-\frac{(1+\rho)(2\tau+1)}{\ell}}.
\end{equation}

 Next, one can estimate the rest new error terms as in subsection (\ref{sectionerror}): the symplectomorphism
 \begin{equation*}
  \|\Phi_m-id\|_{\mathbb{T}^n_{s_{m+1}}\times\mathbb{C}^{2d}}\leq C(m+1)\epsilon_m^{\frac{1-\rho}{2}},
 \end{equation*}
\begin{equation*}
\|\partial(\Phi_m-id)\|_{\mathbb{T}^n_{s_{m+1}}\times\mathbb{C}^{2d}}\leq C(m+1)\epsilon_m^{\frac{1+\rho}{4}}, \ \ \ \ \ \omega\in\mathcal{D}_{m+1},
\end{equation*}
\begin{equation*}
  \|\partial_{\omega_l}(\Phi_m-id)\|_{\mathbb{T}^n_{s_{m+1}}\times\mathbb{C}^{2d}}\leq C(m+1)\epsilon_m^{\rho},\ \ \ \ l=1,\cdots,n,
\end{equation*}
and the new perturbation term $q_{m+1}\in \mathcal{Q}^{\Re}_{s_{m+1},\mathcal{D}_{m+1}}$ with
 \begin{equation*}
  [q_{m+1}]_{s_{m+1},\mathcal{D}_{m+1}}\leq C(m+1).
\end{equation*}
Moreover,
\begin{align}\label{partialphi}
   \nonumber       \|\partial\widetilde{\Phi}_m\|_{\mathbb{T}^n_{s_{m+1}}\times\mathbb{C}^{2d}} & =\|(\partial\Phi_0\circ\Phi_2\cdots\circ\Phi_m)(\partial\Phi_2\circ\Phi_3\cdots\circ\Phi_m)
          \cdots(\partial\Phi_m)\|_{\mathbb{T}^n_{s_{m+1}}\times\mathbb{C}^{2d}}  \\
           & \leq \prod_{j=0}^m(1+\epsilon_j^{\frac{1+\rho}{4}})\leq 2.
\end{align}
It follows immediately that
$$\widetilde{\Phi}_m:\mathbb{T}^n_{s_{m+1}}\times\mathbb{C}^{2d}\rightarrow\mathbb{T}^n_{\sigma_{m+1}}\times\mathbb{C}^{2d}.$$
Therefore
$$q_{0, m+1}\circ\widetilde{\Phi}_m\in\mathcal{Q}^{\Re}_{s_{m+1},\mathcal{D}_{m+1}},\ \ \ \ \ \
[q_{0, m+1}\circ\widetilde{\Phi}_m]_{s_{m+1},\mathcal{D}_{m+1}}\leq C(m+1).$$
As a result, $q'_{m+1}\in \mathcal{Q}^{\Re}_{s_{m+1},\mathcal{D}_{m+1}}$ with
 \begin{equation*}
  [q'_{m+1}]_{s_{m+1},\mathcal{D}_{m+1}}\leq C(m+1).
\end{equation*}

The lemma is finished.
\end{proof}

\subsection{Transforming to the limit and the proof of Theorem \ref{kam}}

Let $\mathcal{D}^{\star}_{\epsilon}=\bigcap_{m\geq 0}\mathcal{D}_m$. This is a closed set satisfying
\begin{equation*}
  \meas(\mathcal{D}\setminus\mathcal{D}^{\star}_{\epsilon})\leq\sum_{m\geq 0}\gamma_m^{\frac{1}{2}}\leq 4\gamma_0^{\frac{1}{2}}.
\end{equation*}
 For $\omega\in\mathcal{D}^{\star}_{\epsilon}$,
\begin{align*}
  \|\widetilde{\Phi}_{m+1}-\widetilde{\Phi}_m\|&=\|\Phi_0\circ\cdots\circ\Phi_m\circ(\Phi_{m+1}-id)\|\\
  &\leq \prod_{j=0}^{m}(1+C(j+1)\epsilon_j^{\frac{1-\rho}{2}}) C(m+2)\epsilon_{m+1}^{\frac{1-\rho}{2}}\\
  &\leq 2^{m+1}C(m+2)\epsilon_{m+1}^{\frac{1-\rho}{2}}.
\end{align*}
Therefore $(\tilde{\Phi}_m)_m$ is a Cauchy sequence, and thus when $m\rightarrow 0$, the mappings $(\tilde{\Phi}_m)_m$ converge to a limit mapping $\Phi_{\epsilon}:\mathbb{T}^n\times\mathbb{C}^{2d}\rightarrow\mathbb{T}^n\times\mathbb{C}^{2d}, \omega\in\mathcal{D}_{\epsilon}$. Moveover, since the convergence is uniform in $\mathcal{D}_{\epsilon}$ and $\theta\in\mathbb{T}^n$, $\Phi_{\epsilon}$ is analytic on $\theta$  and $C^1$ in $\omega$, and
\begin{align*}
  \|\Phi_{\epsilon}-id\|_{\mathbb{T}^n\times\mathbb{C}^{2d}}
  &=\|\sum_{m+0}^{\infty}(\widetilde{\Phi}_{m+1}-\widetilde{\Phi}_m)+\Phi_0-id\|_{\mathbb{T}^n\times\mathbb{C}^{2d}}\\
  &\leq \sum_{m+0}^{\infty}2^{m+1}C(m+2)\epsilon_{m+1}^{\frac{1-\rho}{2}}+C(1)\epsilon_0^{\frac{1-\rho}{2}}\leq \epsilon_0^{\frac{1}{4}}.
\end{align*}
By construction, the mapping $\tilde{\Phi}_m$ transforms the original Hamiltonian $h_0+q_0$ to $h_{m+1}+\epsilon_{m+1}q_{m+1}$. When $m\rightarrow\infty$, by (\ref{qj}), it appears $q_m\rightarrow 0$, and by (\ref{nj}) one gets $N_m\rightarrow N_{\infty}=N_0+\sum_{m=1}^{\infty}\widetilde{N}_m$ which is a Hermitian matrix being $C^1$ with respect to $\omega\in\mathcal{D}_{\epsilon}$. Denoting $h_{\infty }(\theta,z,\bar{z})=\langle z,N_{\infty}(\omega)\bar{z}\rangle+e_{\infty}$, it has been proved that
\begin{equation*}
 ( h_0+q_0)\circ\Phi_{\epsilon}=h_{\infty}.
\end{equation*}
Furthermore, for any $\omega\in \mathcal{D}_{\epsilon}$, by using (\ref{nj}), one has
\begin{equation*}
 \|e_{\infty}\|,\  \|N_{\infty}-N_0\|\leq\sum_{m=0}^{\infty}C(m)\epsilon_m\leq \epsilon_0^{\frac{1}{2}}
\end{equation*}
and thus the eigenvalues of $N_{\infty}(\omega)$, denoted $v_j^{\infty}(\omega)$, satisfy (\ref{v}).

It remains to explicit the affine symplectomorphism $\Phi_{\infty}$. At each step of the KAM procedure it appears
\begin{equation*}
  \Phi_{j-1}(\theta,z,\bar{z})=(\theta,\phi_{j-1}(\theta,z,\bar{z})),
\end{equation*}
 and therefore for $\omega\in\mathcal{D}_{\epsilon}$ and $(\theta,z,\bar{z})\in\mathbb{T}^n\times\mathbb{C}^{2d}$,
 \begin{equation*}
  \Phi_{\epsilon}(\theta,z,\bar{z})=(\theta,\phi_{\epsilon}(\theta,z,\bar{z}))
\end{equation*}
where $g_{\epsilon}=\lim_{m\rightarrow \infty}g_0\circ\cdots\circ g_m$ and $\phi_{\epsilon}=\lim_{m\rightarrow \infty}\phi_0\circ\cdots\circ\phi_m$.

It is useful to go back to real variables $(x,\xi)$. More precisely write each Hamiltonian $\epsilon_mf_m$ constructed in the KAM iteration in the variables $\tilde{u}=\left(
             \begin{array}{c}
               x \\
               \xi \\
             \end{array}
           \right)$:
\begin{equation*}
  \epsilon_mf_m(\theta,x,\xi)=\frac{1}{2}\langle \tilde{u},\tilde{f}_m^1(\theta)\tilde{u}\rangle+\langle \tilde{f}_m^2(\theta),\tilde{u}\rangle
\end{equation*}
with
\begin{equation*}
  \tilde{f}_m^1(\theta)=\epsilon_m\left(
                                     \begin{array}{cc}
                                       -(F_m^{zz}+{F}^{\bar z\bar z}_m)+\frac{1}{2}\left((F^{z\bar{z}}_m)^{\top}+F^{z\bar{z}}_m \right) & -\text{i}(F_m^{zz}-{F}^{\bar z\bar z}_m)+\frac{\text{i}}{2}\left((F^{z\bar{z}}_m)^{\top}-F^{z\bar{z}}_m \right) \\
                                       -\text{i}(F_m^{zz}-{F}^{\bar z\bar z}_m)+\frac{\text{i}}{2}\left((F^{z\bar{z}}_m)^{\top}-F^{z\bar{z}}_m \right) & F_m^{zz}+{F}^{\bar z\bar z}_m+\frac{1}{2}\left((F^{z\bar{z}}_m)^{\top}+F^{z\bar{z}}_m \right) \\
                                     \end{array}
                                   \right),
\end{equation*}
\begin{equation*}
  \tilde{f}_m^2(\theta)=\frac{\sqrt{2}}{2}\epsilon_m\left(
                                    \begin{array}{c}
                                      -\text{i}(F_m^z(\theta)-{F}_m^{\bar{z}}(\theta)) \\
                                      F_m^z(\theta)+{F}_m^{\bar z}(\theta) \\
                                    \end{array}
                                  \right).
\end{equation*}
From the structures of $f_m$ and $q_m$, $\tilde{f}_m^1(\theta)$ and $\tilde{f}_m^2(\theta)$ are real-valued for any $(\theta,\omega)\in\mathbb{T}^n\times\mathcal{D}_{m+1}$. Denote $JB_m(\theta)=\frac{1}{2}(\tilde{f}_m^1(\theta)+(\tilde{f}_m^1(\theta))^{\top})$ and $JU_m(\theta)=\tilde{f}_m^2(\theta)$ with $J:=\left(
                                               \begin{array}{cc}
                                                 0 & I_d \\
                                                 -I_d & 0 \\
                                               \end{array}
                                             \right)
$. Then $(JB_m)^{\top}=JB_m$,  which means $B_m$ is a Hamiltonian matrix for any $(\theta,\omega)\in\mathbb{T}^n\times\mathcal{D}_{m+1}$. Moreover, $B_m$ and $\tilde{f}_m^2$ have the bounds $C(m+1)\epsilon_m^{\frac{1-\rho}{2}}$.
Let $\tilde{u}(t)$ be the integral curve of the Hamiltonian vector field  $\mathbf{X}_{\epsilon_0 f_0}$ in the symplectic space $(\mathbb{R}^{2d}, \text{d}x\wedge \text{d}\xi)$ with the initial value  $\tilde{u}_0=\tilde{u}(0)=\left(\begin{array}{c}
  x \\
  \xi
\end{array}\right)\in \mathbb{R}^{2d}$.
Then $\tilde{u}(t), t\in[0,1]$ satisfies the  integral equation
\begin{equation*}
  \tilde{u}(t)-\tilde{u}_0=B_m(\theta)\int_{0}^t (\tilde{u}(s)-\tilde{u}_0)\text{d}s  + (B_m(\theta)\tilde{u}_0+U_m(\theta))t
\end{equation*}
and therefore
\begin{equation*}
  \tilde{u}(t)=e^{B_m(\theta)t}\tilde{u}_0+\int_{0}^te^{(t-s)B_m(\theta)}U_m(\theta)\text{d}s.
\end{equation*}
Particularly for $t=1$,  $\phi_m$ written in  the real variables has the form
\begin{equation*}
  \phi_m(\theta,x,\xi)=e^{B_m(\theta)}(x,\xi)+T_m(\theta), \ \ T_m(\theta)=\int_{0}^1e^{(1-s)B_m(\theta)}U_m(\theta)\text{d}s,
\end{equation*}
and one has the following lemma.
\begin{lem}
There exists a sequence of Hamiltonian matrices $A_m(\theta)$ and vectors $V_m(\theta)\in \mathbb{R}^{2d}$ such that
\begin{equation*}
  \phi_0\circ\cdots\circ\phi_m(x,\xi)=e^{A_m(\theta)}(x,\xi)+V_m(\theta),\  \ \ \ \ \forall (x,\xi)\in \mathbb{R}^{2d}.
\end{equation*}
Furthermore, there exist a Hamiltonian matrix $A_{\infty}(\theta)$ and a vector $V_{\omega}(\infty)\in\mathbb{R}^{2d}$ such that
\begin{equation*}
  \lim_{m\rightarrow +\infty}e^{A_m(\theta)}=e^{A_{\infty}(\theta)},\ \ \ \ \ \lim_{m\rightarrow+\infty}V_m(\theta)=V_{\infty}(\theta),
\end{equation*}
\begin{equation*}
  \sup_{\theta\in\mathbb{T}^n,\omega\in\mathcal{D}_{\epsilon}}\|A_{\infty}(\theta)\|\leq \epsilon_0^{\frac{1}{4}}, \ \ \ \ \  \sup_{\theta\in\mathbb{T}^n,\omega\in\mathcal{D}_{\epsilon}}\|V_{\infty}(\theta)\|\leq \epsilon_0^{\frac{1}{4}},
\end{equation*}
and for  $(\theta,\omega)\in\mathbb{T}^n\times\mathcal{D}_{\epsilon}$,
\begin{equation*}
 \phi_{\infty}(x,\xi)=e^{A_{\infty}(\theta)}(x,\xi)+V_{\infty}(\theta),\ \ \ \ \ \forall (x,\xi)\in\mathbb{R}^{2d}.
\end{equation*}
\end{lem}

\begin{proof}
Recall that $\phi_m=e^{B_m}+T_m$ where $T_m$ is a translation by the vector $T_m$ with the estimates $\|B_m\|,\ \|T_m\|\leq C(m+1)\epsilon_m^{\frac{1-\rho}{2}}$. So one has $e^{B_m}=I_{2d}+S_m$ with $\|S_m\|\leq C(m+1)\epsilon_m^{\frac{1}{3}}$. Then $\prod_{0\leq m\leq j}e^{B_m}$ has the form
\begin{equation*}
  \prod_{0\leq m\leq j}(I_{2d}+S_m)=I_{2d}+S_0+\cdots+S_j+ S_0S_1+\cdots+S_{j-1}S_j+\cdots+S_0S_1\cdots S_j
\end{equation*}
with
\begin{equation*}
 \| \prod_{0\leq m\leq j}e^{B_m}\|\leq 1+2\|S_0\|+2\| S_0S_1\|+\cdots+2\|S_0S_1\cdots S_j\|\leq 1+C(1)\epsilon_0^{\frac{1}{3}}.
\end{equation*}
Therefore  $\prod_{0\leq m\leq j}e^{B_m}=I_{2d}+M_j$ with $\|M_j\|\leq C(1)\epsilon_0^{\frac{1}{3}}$, and the infinite product $\prod_{0\leq m\leq +\infty}e^{B_m}$ is convergent. Denote  $\prod_{0\leq m\leq +\infty}e^{B_m}=I_{2d}+M$. Then $\|M\|\leq C(1)\epsilon_0^{\frac{1}{3}}\leq \epsilon_0^{\frac{1}{4}}$.
Since $M_j$ has a small norm,  $A_j:=\log(I_{2d}+M_j)$ is well defined. Furthermore, by construction $I_{2d}+M_j\in \text{Sp}(2d)$ (symplectic group) and therefore $A_j\in \text{sp}(2d)$, i.e., $A_j$ is a Hamiltonian matrix for any  $1\leq j\leq +\infty$.
By conclusion,
\begin{equation*}
  \phi_0\circ\phi_1\circ\cdots\phi_j(x,\xi)=e^{A_j}(x,\xi)+V_j,
\end{equation*}
with $V_{j+1}=e^{A_j}T_{j+1}+V_j$, $V_0=T_0$ and the estimations
\begin{equation*}
  \|V_{j+1}-V_j\|\leq C\|T_{j+1}\|\leq C(j+2)\epsilon_{j+1}^{\frac{1}{3}}.
\end{equation*}
Finally  $\lim_{j\rightarrow+\infty}V_j=V_{\infty}$ exists.
\end{proof}

\subsection*{Acknowledgment}
The author would like to thank Prof. Xiaoping Yuan for his helpful suggestions.

\bibliographystyle{abbrv} 
\bibliography{reducibility}

\end{document}